\documentclass[12pt]{article}
\usepackage[utf8]{inputenc}
\usepackage[english,russian]{babel}
\usepackage{graphicx}
\usepackage{amsmath,amssymb,latexsym,amsfonts}

\usepackage[top=45pt, bottom=30mm, left=45pt, right=30mm]{geometry}

\newtheorem{theorem}{Theorem}

\newtheorem{lem}[theorem]{Lemma}
\newtheorem{prop}[theorem]{Proposition}
\newtheorem{cor}[theorem]{Corollary}

\newtheorem{rem}{Замечание}

\newenvironment{proof}{\textsc{Proof}.} 

\numberwithin{equation}{section}
\numberwithin{theorem}{section}

\newcommand{\1}{\hspace{1pt}}

\newcommand{\su}{\subset}

\newcommand{\mG}{\mathbb{G}}

\newcommand{\mg}{\mathbf{g}}

\newcommand{\de}{\delta}

\newcommand{\mmp}{ \mathbf{p} }

\newcommand{\mn}{ \mathbf{n} }

\newcommand{\mq}{ \mathbf{q} }

\newcommand{\mN}{ \mathbf{N} }
\newcommand{\mbN}{ \mathbb{N} }

\begin{document} 

\begin{center}
{\bf On multiple null-series in the Walsh system, M- and U-sets}
\end{center}

\begin{center}
A.~D.~Kazakova, M.~G.~Plotnikov
\end{center}

\begin{center}

A family of $M$-sets and null-series for the d-dimensional Walsh system 
are constructed if we consider convergence over rectangles, cubes, or iterated convergence. Non-empty portions of the constructed $M$-sets are also $M$-sets. The question of the rate of convergence to zero of the coefficients of zero-series that realize the constructed $M$-sets is studied, and it is shown how to modify the construction of the latter to turn them into $U$-sets.
\end{center}

Bibliography: 46 titles.


\section*{Introduction}

In 1916, D.E.~Men'shov  \cite{men'shov} constructed the first nontrivial trigonometric series 
that converges to zero almost everywhere.
Such series are called {\it null-series}, 
and the set outside of which a null series converges to zero is an $M$-set. 
More precisely, a set $A$ is called an {\it $M$-set} 
for some system of functions, if there exists a series 
with respect to this system that realizes it, 
i.e. a series converging to zero outside $A$ whose coefficients are not all zero.
A set that is not an $M$-set is called a {\it $U$-set} 
(or alternatively, a {\it set of uniqueness}).

Numerous studies have demonstrated that determining whether a given set belongs 
to the $U$-set or $M$-set class for the trigonometric system is a highly nontrivial problem. This depends not only on the metric and topological structure of the set, but also on its arithmetic properties.
Thus, to determine whether a symmetric closed set $F_\zeta$ with constant ratio $\zeta$ (a Cantor-type set) is a $U$-set for the trigonometric system, one needs to investigate the arithmetic properties of the parameter $\zeta$ 
(the Salem--Zygmund--Bari--Pyatetsky-Shapiro theorem \cite{zygmund-2002}, \cite{bari-1961}). 
A constructive criterion for belonging to the family of $U$-sets 
cannot be found even 
for the class of closed sets.
See \cite{kechris-louveau-1987}. 

For some recent results concerning uniqueness problems,
including for other systems of functions,
see \cite{kholshchevnikova-2019, kozma-olevskii-2020, gevorkyan-2020, plotnikov-2020,  
skvortsov-2021, wronicz-2021, lukomskii-2021, plotnikov-2022, 
gevorkyan-2023, gevorkyan-2024, keryan-khachatryan-2024, kazakova-plotnikov-2025-1, kazakova-plotnikov-2025-2}.   

This paper studies problems related to null-series, $M$-sets, and $U$-sets for the multidimensional Walsh system.
In the one-dimensional case, the existence of $M$-sets of zero measure for the Walsh system was established by 
A.A.~Shneider, J.E.~Coury, and F.~Schipp \cite{shneider-1949, schipp-1969, coury-1970} 
while the first construction of a perfect
$M$-set of zero measure for the Walsh system was given by V.A.~Skvortsov \cite{skvortsov-1976}. 
The papers \cite{skvortsov-1977-1, gevorkian-1988} are devoted to the question of
the behavior of the coefficients of zero-series. 
In particular, G.G.~Gevorkyan proved \cite{gevorkian-1988} 
that for every positive monotonically tending to zero sequence not in $l_2$, there exists a Walsh null-series
whose coefficients are majorized by this sequence. 
In~\cite{skvortsov-1977-2}, V.A.~Skvortsov constructed an example of a perfect $M$-set with zero Hausdorff $h$-measure for every $h > 0$. Thus, a set that is ``thin'' in this specific sense turns out to be an $M$-set. 

Later results established the existence of null-series for a more general class of multiplicative periodic orthonormal Vilenkin systems. Each such system is defined by a sequence of integers $(p_n \ge 2)$. (V.A.~Skvortsov~\cite{skvortsov-1979} --- the case $\sup p_n < +\infty$, 
I.I.~Tuzikova~\cite{tuzikova-1985} --- the case $\lim\limits_{n \to \infty} 
\inf\limits_{m \ge n} p_m < +\infty$, 
N.A.~Bokaev and M.A.~Nurkhanov~\cite{bokaev-nurhanov-1993} --- the general case).

Given one-dimensional null-series, multidimensional null-series can be constructed \cite{kholshchevnikova-2013}:  
if $\sum\limits_n a_n\phi_n(x_1) $ and $\sum\limits_m b_m\phi_m(x_2)$ are 1-dim null-series, 
$F_1$ and $F_2$ are their corresponding $M$-sets, 
then the double series $\sum\limits_{n,m} a_n b_m \phi_n(x_1) \phi_m(x_2)$ 
is a null-series with respect to the function system $\{ \phi_n (x_1) \phi_m (x_2) \}$ 
under convergence over rectangles
or over cubes, 
and the set
\begin{equation} \label{F}
F = (F_1 \times [0, 1]) \cup ([0, 1] \times F_2)
\end{equation} 
is an $M$-set. 

For multidimensional Walsh series, broad classes of uniqueness sets under rectangular convergence were obtained by S.F.~Lukomsky, L.D.~Gogoladze, and T.A.~Zherebyeva
\cite{lukomskii-89, gogoladze-2008, zherebeva-2009} 
(see also~\cite{skvortsov-tulone-2015}, where such classes are essentially contained), 
for cube and $\lambda$-convergence by M.G.~Plotnikov and S.F.~Lukomsky   
\cite{plotnikov-07-1, plotnikov-07-2, plotnikov-2010, plotnikov-2017, lukomskii-2021}. 
In another work \cite{Lukomskii-1992} S.~F.~Lukomskii, 
using the construction of V.~A.~Skvortsov \cite{skvortsov-1976},
proved the existence of a set that is a $U$-set under rectangular convergence, 
but an $M$-set under cube convergence and $\lambda$-one.
In all works concerning cube convergence and $\lambda$-convergence, the Cartesian power $d$ 
of the Cantor dyadic group $\mathbb{G}$ was considered as the domain of $d$-dim Walsh functions.
If we consider the $d$-dim Walsh system on the unit cube $[0,1]^d$ 
rather than on the set $\mathbb{G}^d$, then it is not even known whether $\emptyset$ is a $U$-set. 
An analogous question remains open for the the $d$-dim trigonometric system as well. 
For multiple Walsh and trigonometric series (under both rectangular and cubic convergence), the following interesting question remains open: are all sets of positive measure necessarily $M$-sets?
This question has been repeatedly posed in several works, for example in \cite{ash-freiling-rinne-1993, 
kholshchevnikova-2002}. 

In \S~\ref{S:M-Set} of this paper, 
we construct a new class of $M$-sets for the $d$-dim Walsh system along with their corresponding null-series, 
considering not only rectangular and cubic convergence but also iterated one.
 
In \S~\ref{S:Main-results} it is proved 
(Theorems~\ref{T:Main-1} and \ref{T:Main-1-1}) 
that the constructed $M$-sets and null-series 
are in fact such.

Remarkably, the $M$-sets \eqref{F} have uncountably many sections along each coordinate axis with full one-dimensional Lebesgue measure (when considering Walsh functions as $\{\phi_n\}$). 
The $M$-sets constructed in \S~\ref{S:M-Set} for the $d$-dimensional Walsh system not only themselves have $d$-dimensional zero measure, but also every their section by a $k$-dimensional plane parallel to the coordinate axes 
has $k$-dimensional zero measure whenever $k = 1, \ldots, d-1$.

In Theorem~\ref{T:Portion-2} it is shown that non-empty portions of the constructed $M$-sets will also be $M$-sets, 
and the zero series realizing them are explicitly constructed. 
This yields a positive answer to the question posed in  \cite{kholshchevnikova-2013} by N.\,N.~Kholshchevnikova  
regarding the existence of $M$-sets contained, for instance, within the cube $[0,1/2]^d$.
For rectangular convergence, the part of Theorem~\ref{T:Portion-2} concerning $M$-sets 
can alternatively be proved using Lemmas 1 and 3 from another work of N.\,N.~Kholshchevnikova~\cite{kholshchevnikova-2002}.

Theorem~\ref{T:Main-2-2-2} establishes, in the spirit of 
G.\,G.~Gevorkyan's work \cite{gevorkian-1988} 
for the 1-dim Walsh system, 
that the coefficients of the constructed null-series cannot be significantly 
reduced while still realizing the original $M$-set.

Interestingly, the structure of the constructed $M$-sets becomes 
highly transparent when described using the ``graphs'' of multidimensional Walsh functions $W_{\mathbf{n}}$. 
Such sets $F$ arise as intersections of ``layers'' $F_s$, 
where each layer is obtained by covering the domain of the Walsh functions with several graphs.
Because of the use of ``graphs'' of different functions,  
the set $F$ is in some sense  
strongly non-symmetric.
For details, see \S~\ref{S:M-Set}.
As soon as we make the set $F$ more symmetric 
by covering the domain with ``graphs'' of a single Walsh function, 
we immediately obtain a $U$-set (Theorem~\ref{T:U-2} 
and Corollary~\ref{Cor:U-2}).

The structure of the constructed multiple Walsh null-series also becomes highly transparent when, instead of series, 
one considers finitely additive set functions of dyadic type (so-called \emph{quasi-measures}) 
supported on the constructed $M$-sets.
This approach is motivated by the fact that the sets of multiple Walsh series and quasi-measures are isomorphic as linear spaces, where each such series is the Walsh-Fourier series of the quasi-measure it generates.
For details, see, for example, \cite{MP-EMJ-2019}.   
 
The rest of the paper is organized as follows.
\S~\ref{S:Prelim-1} contains the main definitions and general auxiliary results. Technical lemmas about the sets and Walsh series constructed in Section~\ref{S:M-Set} are presented in Section~\ref{S:Prelim-3}.

In our paper, we use some one-dimensional ideas from \cite{skvortsov-1976, yoneda-1982} 
as well as $d$-dim techniques from \cite{plotnikov-07-1, plotnikov-07-2, plotnikov-2010}.


\section{Preliminaries} 
\label{S:Prelim-1}

\subsection{Notations} 

We write $:=$ to denote equality by definition.

$\# A$ denotes the cardinality of (a finite set) $A$. 

$\mathbb{C}$ denotes the set of complex numbers,
$\mathbb{N}$ the set of positive integers; 
$\mathbb{N}_0 := \mathbb{N} \cup \{ 0 \}$. 

\[ 
\mathrm{I} (A) 
:= 
\begin{cases} 
1  &\text{if statement $A$ is true}, 
\\ 
0  &\text{if statement $A$ is false}, 
\end{cases}
\] 

$\delta_m^p := \mathrm{I} ( m=p )$ is the \emph{Kronecker delta}.  

We denote by $n_k$ the \emph{binary coefficients} of $n \in \mathbb{N}_0$, 
obtained from its \emph{binary expansion}  
$n = \sum_{k=0}^\infty n_k 2^k$ where $n_k = 0 \vee 1$.

We use the notation $a:b$ for the set $\{ a, a+1, \ldots, b-1, b \}$. 

Multiplication of a vector by a scalar is understood in the standard sense.

Throughout this paper, let $d \ge 2$ be a fixed natural number. 

We write $\mathbf{0}$ for the $d$-dimensional zero vector $(0, \ldots, 0)$, 
$\mathbf{1}$ denotes the $d$-dimensional all-ones vector $(1, \ldots, 1)$,  
$\mathbf{g} = ( g^1, \ldots, g^d )$, $\mathbf{h} = ( h^1, \ldots, h^d )$.
The notation $\mathbf{g} < \mathbf{h}$ ($\mathbf{g} \le \mathbf{h}$) means that $g^j < h^j$ ($g^j \le h^j$) 
for all $j \in 1:d$.  

 $B_k := 
\big\{ 
    \mn \in \mN^d 
    \colon 
    2^k \mathbf{1} \le \mn < 2^{k + 1} \mathbf{1}   
\big\}$. 


\subsection{Basic Definitions and Preliminary Facts}


\subsubsection{} 

The dyadic group $\mathbb{G} = \mathbb{G}_2$ is defined as the direct sum 
of countably many copies of the cyclic group $\mathbb{Z}_2$ (each endowed with the discrete topology) equipped with the product topology (also called the Tikhonov topology).
The zero element of the group $\mathbb{G}$ and inverse elements are defined in the obvious way. 
The elements of $\mathbb{G}$ can be conveniently represented 
as sums of convergent in the topology of $\mathbb{G}$ formal series
\begin{equation}
\label{Eq:PG-El}
\bigoplus\limits_{ k = 0 }^\infty 
g_k e_k, 
\quad
g_k \in \{ 0, 1 \},
\end{equation} 
or as the series themselves.
Here, $e_k$ are the $k$-th generators of the group $\mathbb{G}$, satisfying $2e_k = 0$, 
and the group operation $\oplus$ is applied componentwise to the elements in \eqref{Eq:PG-El}.

For each $d \in \mathbb{N}$, 
the set $\mathbb{G}^d = (\mathbb{G}_2)^d$ is a topological abelian group with the addition operation
\[
\mathbf{g} \oplus \mathbf{h} := (g^1 \oplus h^1, \ldots, g^d \oplus h^d),
\]
where the zero element and additive inverses are defined in the natural way. 
We use the notation $\oplus$ for the group operation on both $\mathbb{G}$ and $\mathbb{G}^d$, 
as this convention will not lead to any confusion.
A base for the topology of $\mathbb{G}$ is formed by the cosets
of subgroups 
$
\left\{ \bigoplus\nolimits_{t=k+1}^\infty g_t e_t \right\}$,
which are called \emph{dyadic intervals} of rank $k$ and are often numbered as follows:
\[ 
\Delta^{(k)}_m 
:= 
\left\{ 
    \bigoplus\nolimits_{ t = 0 }^\infty 
    g_t e_t
    \colon 
    g_t = m_{k - 1 - t}, \;\; t \in [ 0, k )  
\right\}, 
\quad 
k \in \mbN_0, 
\quad 
m \in 0 : 2^k - 1, 
\] 
$m_t$ is the binary coefficients of the number $m$. 

A base for the topology of $\mathbb{G}^d$ is formed by $d$-dim  
\emph{dyadic cubes} (of rank $k$)
\begin{equation}
\label{Eq:P-adic-Cube}
\Delta^{(k)}_{\mathbf{m}} 
:= 
\Delta^{(k)}_{m^1}  
\times 
\ldots 
\times 
\Delta^{(k)}_{m^d}, 
\qquad 
k \in \mathbf{N}_0, 
\quad
\mathbf{m} \in ( 0 : 2^k - 1 )^d, 
\end{equation} 
each of which is clopen in the topology of $\mathbb{G}^d$ 
and represents a certain coset of the subgroup $\Delta^{(k)}_{\mathbf{0}}$. 
We say that the cube \eqref{Eq:P-adic-Cube} is of rank $k$.  
We sometimes write simply $\Delta^{(k)}$ to denote an arbitrary dyadic cube of rank $k$. 
Every dyadic cube of rank $k$ is partitioned into $2^d$ 
({\it adjacent}) dyadic cubes of rank $k+1$:
\[ 
\Delta^{(k)}_{\mathbf{m}} 
= 
\bigsqcup\limits_{ \boldsymbol{\sigma} \in \{ 0, 1 \}^d } 
\Delta^{(k+1)}_{ 2 \mathbf{m} + \boldsymbol{\sigma} }. 
\]

The mapping $F$ that associates to an element \eqref{Eq:PG-El} the series sum 
$\sum_{k=0}^\infty g_k 2^{-k-1}$ is a bijection modulo a countable set, 
mapping $\mathbb{G}$ onto $[0,1]$  
and sending $\Delta^{(k)}_m$ to the closed intervals 
$\left[ m 2^{-k}, (m+1)2^{-k} \right] \subset [0,1]$ (see \cite{Schipp-and-Co}).

The coordinate-wise extension $(g^1,\ldots,g^d) \stackrel{F}{\mapsto} (F(g^1),\ldots,F(g^d))$ 
is a bijection modulo a null set between $\mathbb{G}^d$ and $[0,1]^d$, 
mapping the dyadic cubes \eqref{Eq:P-adic-Cube} 
to the cubes $\mathop{\times}\limits_{l=1}^d \left[m^l 2^{-k}, (m^l+1)2^{-k}\right]$.


By the \emph{measure} $\mu$ on the group $\mG^d$ we mean the normalized ($\mu ( \mG^d ) = 1$) Haar measure, 
which is defined on all Borel subsets of the group $\mG^d$ and invariant 
with respect to shifts and transformations converting $H$ to $H^{-1}$. 
We have  $\mu ( \Delta^{(k)}_{\mathbf{m}} ) = 2^{-kd}$. 
For the one-dimensional case, see  \cite{Schipp-and-Co}. 


\subsubsection{}

On the group $\mathbb{G}$, the \emph{Walsh functions} in the Paley numeration  
are defined as
$W_n (g) = \prod_{k=0}^{\infty}
( - 1 )^{g_k n_k}$, where $n \in \mathbb{N}_0$ and $g$ is an element of $\mathbb{G}$ of the form \eqref{Eq:PG-El}.  
For $n < 2^k$ the Walsh function $W_n$ takes a constant value  
\begin{equation}
\label{Eq:WF-DI-1} 
=: W_n ( \Delta^{(k)}_m ) 
= \sum\limits_{j=0}^{k-1} n_j m_{k-1-j}  
\end{equation} 
on $\Delta^{(k)}_m$. 
\emph{$d$-\hspace{0pt}dim Walsh functions} 
$W_{\mathbf{n}}$
\begin{equation}
\label{Eq:VC-F}
W_{\mathbf{n}} ( \mathbf{g} ) 
:= 
\prod\limits_{l=1}^d
W_{n^l} (g^l), 
\quad 
\mathbf{n} \in ( \mathbb{N}_0 )^d, 
\quad 
\mathbf{g} \in \mG^d, 
\end{equation} 
form an orthonormal system in $L^2 ( \mG^d, \mu )$. We have  
\[ 
W_{ \mathbf{n} } ( \mathbf{g} ) 
W_{ \mathbf{n} } ( \mathbf{h} ) 
= 
W_{\mathbf{n}} ( \mathbf{g} \oplus \mathbf{h} ) 
\quad 
\text{for all $\mathbf{n}$, $\mathbf{g}$, $\mathbf{h}$}. 
\] 


{\it $d$-dimensional Walsh series} on $\mG^d$ is defined by 
\begin{equation}
\label{Eq:WSer}
\sum_{ \mathbf{n} \in ( \mbN_0 )^d } 
a_{ \mathbf{n} } W_{ \mathbf{n} } ( \mathbf{g} ), 
\quad
a_{ \mathbf{n} } \in \mathbb{C},  
\end{equation} 
while its the $\mN$-th \emph{rectangular partial sums} at a point $\mathbf{g}$ is defined by 
\begin{equation}
\label{Eq:PartSums}
S_{\mN} ( \mathbf{g} )
:= 
\sum_{ \mathbf{n} < \mN }  
a_{ \mathbf{n} } 
W_{\mathbf{n}} (\mathbf{g}), 
\quad 
\mN \in \mbN^d.
\end{equation} 
Partial sums $S_{\mN}$ with indices $\mathbf{N} = N \mathbf{1}$  
are called \emph{cubic} and denoted simply by $S_N$.

For $\mathbf{n}$ and $\mathbf{N}$ satisfying $\mathbf{n} < 2^k \mathbf{1}$ and $\mathbf{N} - \mathbf{1} < 2^k \mathbf{1}$,  
each Walsh function $W_{\mathbf{n}}$ and partial sum $S_{\mathbf{N}}$  
are constant on $\Delta^{(k)}$, taking values  
$=: W_{\mathbf{n}}(\Delta^{(k)})$ and $=: S_{\mathbf{N}}(\Delta^{(k)})$, respectively.

A series \eqref{Eq:WSer} \emph{converges over rectangles} at a point $\mathbf{g}$ 
to a sum $S \in \mathbb{C}$ if 
\[ 
\lim S_{\mN} ( \mathbf{g} ) = S ,
\;\; \text{ $\min \{ N^1, \ldots, N^d \} \to \infty$}, 
\]
\emph{over cubes} if  $\lim\limits_{N \to \infty} S_N ( \mathbf{g} ) = S$, 
and $\lambda$-converges, $\lambda \ge 1$, if
\[
\lim S_{\mN} ( \mathbf{g} ) = S
\;\; \text{for $\min \{ N^1, \ldots, N^d \} \to \infty$
and $\max\limits_{j,k} N^j / N^k \le \lambda$}.
\]
For $\lambda > 1$, $\lambda$-convergence
is weaker than rectangle one and stronger than cube one.
If $(j_1, \ldots, j_d)$ is a permutation of the numbers $1, \ldots, d$, then the \emph{iterated convergence} at the point $\mathbf{g}$ to the value $S$ of the series \eqref{Eq:WSer}, corresponding to this permutation, means that
\begin{equation}
\label{Eq:WSer-Rep-Conv}
\sum_{ n^{j_1} \in \mbN_0 } 
\bigg( 
    \sum_{ n^{j_2} \in \mbN_0 } 
    \bigg( 
        \ldots 
        \bigg( 
            \sum_{ n^{j_d} \in \mbN_0 } 
            a_{ \mathbf{n} } W_{ \mathbf{n} } ( \mathbf{g} ) 
        \bigg)
    \bigg)    
\bigg) 
= S. 
\end{equation} 


\subsubsection{}

Let 
\[ 
D_N (g) 
:= 
\sum_{ n < N }  
W_n (g) 
\] 
be the $N$-th {\it Dirichlet kernel} for the Walsh system. It
is known (see, for example, \cite{golubov-efimov-skvortsov}, formulae (1.4.11) and (1.4.13)) that
\begin{equation} 
\label{Eq:WFS-2}
D_n = D_{2^k} + R_k D_m, 
\quad 
n = 2^k + m, 
\quad 
m \in 1 : 2^k, 
\end{equation} 
$R_k \equiv W_{2^k}$ are the {\it Rademacher functions}; 
\begin{equation} 
\label{Eq:WFS-3}
D_{2^k} (g) 
= 
\begin{cases} 
2^k     & \text{if $g \in \Delta_0^{(k)}$}, 
\\ 
0       & \text{otherwise}.  
\end{cases}  
\end{equation}

The $\mN$-th Dirichlet kernel $D_{\bf N}$ for the system of $d$-dim Walsh
functions is defined as follows:
\begin{equation} 
\label{Eq:WFS-11} 
D_{\mN} (\mathbf{g}) 
:= 
\sum_{ \mn < \mN }  
W_{\mn} (\mathbf{g}) 
= 
\prod\limits_{l=1}^d
D_{N^l} (g^l), 
\quad 
\mN \in \mbN^d.
\end{equation}  


\subsubsection{} 

Let $k \in \mathbb{N}_0$,  
$\mathbf{n}, \mathbf{m} < 2^k \mathbf{1}$. 
As previously noted, 
the Walsh function $W_{\mathbf{n}}$ takes a constant value on $\Delta^{(k)}_{\mathbf{m}}$,  
which we denote by  
$W^{(k)}_{\mathbf{n} \mathbf{m}}$ and  
which is equal to
$\prod_{j=1}^d W^{(k)}_{n^j m^j}$,   
where $W^{(k)}_{n, m} := W_n ( \Delta^{(k)}_m )$ are
the elements of the $k$-th Walsh matrix $W^{(k)}$ 
(see, for example, \cite[Section~1.3]{golubov-efimov-skvortsov}).
It is well known that $W^{(k)}$ is symmetric and
\begin{equation} 
\label{Eq:Walsh-Matrix-1}
W^{(k)} W^{(k)} = W^{(k)} ( W^{(k)} )^T = 2^k E_{2^k}, 
\end{equation}
$E_{2^k}$ is the unit matrix of the order $2^k$. 
From \eqref{Eq:Walsh-Matrix-1} we obtain the equality
\[ 
\sum\limits_{ m^{\prime} < 2^k }
W_{ m m^{\prime} }^{(k)}
W_{ p m^{\prime} }^{(k)} 
= 
2^k \delta_m^p,   
\] 
from which its $d$-dimensional analogue follows:
\begin{equation} 
\label{Eq:Walsh-Matrix-0} 
\begin{split}
\sum\limits_{ \mathbf{m}^{\prime} < 2^k \cdot \mathbf{1} }
& 
W_{ \mathbf{m} \mathbf{m}^{\prime} }^{(k)}
W_{ \mmp \mathbf{m}^{\prime} }^{(k)}  
= 
\sum\limits_{ ( m^{\prime} )^1 < 2^k }
\ldots 
\sum\limits_{ ( m^{\prime} )^d < 2^k }
\prod\limits_{l=1}^d
W^{(k)}_{ m^l ( m^{\prime} )^l }
\prod\limits_{l=1}^d
W^{(k)}_{ p^l ( m^{\prime} )^l }
\\ 
& 
= 
\sum\limits_{ ( m^{\prime} )^1 < 2^k }
W^{(k)}_{ m^1 ( m^{\prime} )^1 }
W^{(k)}_{ p^1 ( m^{\prime} )^1 }
\cdot 
\ldots 
\cdot  
\sum\limits_{ ( m^{\prime} )^d < 2^k }
W^{(k)}_{ m^d ( m^{\prime} )^d }
W^{(k)}_{ p^d ( m^{\prime} )^d }
\\ 
& 
= 
2^{kd} \delta_{\mathbf{m}}^{\mathbf{p}}.  
\end{split}
\end{equation}


\subsubsection{}

\emph{Quasi-measures} on the group $\mathbb{G}^d$ are 
finitely additive set functions  
$\tau \colon \mathcal{B} \to \mathbb{C}$,  
where $\mathcal{B}$ denotes the semiring  
consisting of $\emptyset$ and all dyadic cubes.
Every quasi-measure $\tau$ extends  
to the ring generated by $\mathcal{B}$;  
if $\tau$ is nonnegative,  
it admits an extension to a $\sigma$-additive measure  
on the $\sigma$-algebra of Borel subsets of $\mathbb{G}^d$.

It is easy to check that a set function 
$\tau \colon \mathcal{B} \to \mathbb{C}$ 
is a quasi-measure if and only if
\begin{equation} 
\label{Eq:QM-Char-Prop} 
\tau  
\big( \Delta^{(k)}_{\mathbf{m}} \big) 
= 
\sum\limits_{ \boldsymbol{\sigma} \in \{ 0, 1 \}^d } 
\tau  
\big( \Delta^{(k+1)}_{ 2 \mathbf{m} + \boldsymbol{\sigma} } \big) 
\;\; 
\text{for all admissible $k$ и $\mathbf{m}$}. 
\end{equation}

The set of all quasi-measures is isomorphic as linear space 
to the space of all series \eqref{Eq:WSer};  
the canonical isomorphism 
is given by the mapping that associates to each series \eqref{Eq:WSer}  
a quasi-measure $\tau$,
\begin{equation} 
\label{Eq:Canon-Iso}
\begin{split}
\tau \big( \Delta^{(k)} \big) 
:= 
& 
\sum\limits_{\mn \le 2^k \mathbf{1}} 
\int\limits_{ \Delta^{(k)} } 
a_{\mathbf{n}} W_{\mathbf{n}} ( \mathbf{g} ) \, d \mu ( \mathbf{g} )  
\\ 
= 
& 
2^{-kd} \1 S_{2^k} ( \Delta^{(k)} ),   
\end{split} 
\end{equation} 
which we say is generated by the given series. 
In \eqref{Eq:Canon-Iso}, one may write 
$\mathbf{n} \leq \mathbf{M}$ instead of $\mathbf{n} \leq 2^k\mathbf{1}$, 
provided that $2^k\mathbf{1} \leq \mathbf{M}$.

With a suitable choice of the concept of integral every series \eqref{Eq:WSer} 
is the Fourier--Walsh series 
of the quasi-measure $\tau$ generated by it, 
that is, $a_{\mathbf{n}} \equiv \widehat{\tau}_{\mathbf{n}}$. 
Here $\widehat{\tau}_{\mathbf{n}}$ are the 
\emph{Fourier--Walsh coefficients} of $\tau$,    
\begin{equation}
\label{Eq:Rec-Q-M}
\widehat{\tau}_{ \mathbf{n} }
:= 
\int\limits_{ \mG^d } 
W_{\mathbf{n}} \, d \tau
:=
\sum\limits_{ \mathbf{m} < 2^k \mathbf{1} } 
W_{ \mathbf{n} }  
\big( \Delta^{(k)}_{\mathbf{m}} \big) 
\tau 
\big( \Delta^{(k)}_{\mathbf{m}} \big),   
\end{equation} 
the equality on the right is considered 
when $k$ is sufficiently large for  
$\mathbf{n} < 2^k \mathbf{1}$ to hold.
Conversely, the value of quasi-measure $\tau$ on any dyadic cube  
can be expressed through its Fourier--Walsh coefficients:
\begin{equation} 
\label{Eq:Canon-Iso-add}
\tau ( \Delta^{(k)} ) 
= 
2^{-kd}
\int\limits_{ \Delta^{(k)} } 
\sum_{ \mathbf{n} < 2^k \mathbf{1} } 
\widehat{\tau}_{ \mathbf{n} } W_{ \mathbf{n} } 
d \mu. 
\end{equation} 

The \emph{support} of a quasi-measure $\tau$ 
is defined as the (closed) set $F = \mathbb{G}^d \setminus G$, 
where $G$ is the union of all dyadic cubes 
$\Delta_0$ such that $\tau(\Delta) = 0$ 
for all dyadic cubes $\Delta \subset \Delta_0$. 
Notation: $\mathrm{supp} \, \tau$.

For details, see: \cite[Section~2.3]{MP-EMJ-2019}; \cite{Schipp-and-Co};  
\cite[Chapter~4]{VA-Hung-2004}. 


\subsubsection{} 
\label{Subsub:F-Tau-05}

If a set $F \subset \mG^d$ is closed 
and the series \eqref{Eq:WSer} convergence over cubes to zero on 
$\mG^d \setminus F$, then $\mathrm{supp} \, \tau \subset F$ 
for the quasi-measure $\tau$ generated by this series
For a proof, see, for example, \cite[Lemma~1]{plotnikov-07-2}.

\subsubsection{} 

From formula \eqref{Eq:Rec-Q-M}, it is easy to obtain the following integral representation for the partial sums \eqref{Eq:PartSums} of a series \eqref{Eq:WSer}:
\begin{equation}
\label{Eq:PS-IR} 
\begin{split}
S_{\mN} ( \mathbf{g} ) 
& 
= 
\int\limits_{ \mG^d } 
D_{\mN} ( \mathbf{g} \oplus \mathbf{h} ) \, d \tau ( \mathbf{h} ) 
:=
\sum\limits_{ \mathbf{m} < 2^k \mathbf{1} } 
\tau 
\big( \Delta^{(k)}_{\mathbf{m}} \big) 
D_{\mN}  
\big( \mathbf{g} \oplus \Delta^{(k)}_{\mathbf{m}} \big) 
\\ 
& 
= 
\int\limits_{ \mG^d } 
D_{\mN} ( \mathbf{h} ) \, d \tau ( \mathbf{g} \oplus \mathbf{h} ) 
:=
\sum\limits_{ \mathbf{m} < 2^k \mathbf{1} } 
\tau 
\big( \mathbf{g} \oplus \Delta^{(k)}_{\mathbf{m}} \big) 
D_{\mN}  
\big( \Delta^{(k)}_{\mathbf{m}} \big).  
\end{split} 
\end{equation}
Here $\mathbf{N} < 2^k \mathbf{1}$. 
For details, see, e.g., \cite{plotnikov-2010}, \cite{MP-EMJ-2019}. 


\subsubsection{} 
\label{Subsubs:Tau-F}

To each nonempty closed set $F \subset \mathbb{G}^d$,  
we associate a nonnegative quasi-measure $\tau = \tau_F$  
satisfying: 
$\tau (\mG^d) = 1$; 
equality $\tau (\Delta) = 0$ holds if and only if 
$\Delta \cap F = \emptyset$, where $\Delta$ is a dyadic cube; 
if $\Delta^{(k)}_{\mathbf{m}} \cap F \ne \emptyset$ 
and exactly $M$ of $2^d$ adjacent dyadic cubes $\Delta^{(k+1)}_{2 \mathbf{m}+\boldsymbol{\sigma}}$, 
$\boldsymbol{\sigma} \in \{ 0,1 \}^d$,  
have a nonempty intersection with $F$, 
then 
\begin{equation} 
\label{Eq:Tau-F-111}
\tau ( \Delta^{(k+1)}_{2 \mathbf{m}+\boldsymbol{\sigma}} ) 
:= 
\begin{cases}
\dfrac{ \tau ( \Delta^{(k)}_{\mathbf{m}} ) }{M}, 
& 
\text{if $\Delta^{(k+1)}_{2 \mathbf{m}+\boldsymbol{\sigma}}  
\cap F \ne \emptyset$}, 
\\ 
0 
& 
\text{otherwise}.
\end{cases} 
\end{equation}  
It is easy to check that such a quasi-measure exists
and is unique. 
Obviously, $\mathrm{supp}\, \tau_F = F$. 

A similar construction was used in ~\cite{plotnikov-2004}. 

We restrict the quasi-measure $\tau_F$ to a dyadic cube $\widetilde{\Delta}$  
by considering the set function 
\[ 
\tau_F |_{\widetilde{\Delta}} 
\colon \mathcal{B} \to \mathbb{C}, 
\quad 
\tau_F |_{\widetilde{\Delta}} ( \Delta )
:= 
\tau_F ( \widetilde{\Delta} \cap \Delta ) 
\]
(we use the obvious fact that  
the intersection of two dyadic cubes is either empty  
or is itself a dyadic cube).  

It is easy to check that the function  
$\tau = \tau_F \big|_{\widetilde{\Delta}}$  
satisfies \eqref{Eq:QM-Char-Prop} and is therefore a quasi-measure.  
This quasi-measure is nonnegative and is identically zero  
if and only if $\widetilde{\Delta} \cap F = \emptyset$.

Obviously, $\mathrm{supp} \, \tau_F |_{\widetilde{\Delta}} = F \cap \widetilde{\Delta}$.

 
\subsection{Auxiliary results} 


\begin{prop} 
\label{Prop:WSer-Num}
Suppose the indices of all nonzero terms of the $d$-dim numerical series
\begin{equation}
\label{Eq:WSer-Num}
\sum_{ \mathbf{n} \in \mbN_0^d } 
a_{ \mathbf{n} } 
= 
\sum_{ n^1, \ldots, n^d \in \mbN_0 } 
a_{ n^1, \ldots, n^d }  
\end{equation} 
lie in the set $\{ \mathbf{0} \} \bigsqcup \bigsqcup\limits_{k \in \mathbb{N}_0} B_k$,
where, recall that, $B_k := 
\big\{ 
    \mn \in \mN^d 
    \colon 
    2^k \mathbf{1} \le \mn < 2^{k + 1} \mathbf{1}   
\big\}$,   
and the series itself converges to the sum $S$ over rectangles.  
Then it converges to $S$ in the sense of iterated convergence, 
whatever the order of the iterated summation. 
\end{prop}

\begin{proof} 
We need to show that
\begin{equation}
\label{Eq:WSer-Rep-Conv-1}
\sum_{ n^{j_1} \in \mbN_0 } 
\bigg( 
    \sum_{ n^{j_2} \in \mbN_0 } 
    \bigg( 
        \ldots 
        \bigg( 
            \sum_{ n^{j_d} \in \mbN_0 } 
            a_{ n^1, \ldots, n^d } 
        \bigg)
    \bigg)    
\bigg) 
= S
\end{equation}
for any permutation $(j_1, \ldots, j_d )$ of the numbers $1$, $\ldots$, $d$.

From the conditions of the proposition it follows that
for a fixed $n^{j_1}$, 
only finitely many terms of the series \eqref{Eq:WSer-Num} 
are nonzero.
Consequently, all series in parentheses in \eqref{Eq:WSer-Rep-Conv-1}
become finite sums,
and therefore the expression inside the pair of outer
parentheses is well-defined.

Next, take a natural number $N^1$. 
Then
\begin{equation}
\label{Eq:WSer-Rep-Conv-2}
\begin{split} 
& 
\sum_{ n^{j_1} < N^1 } 
\bigg( 
    \sum_{ n^{j_2} \in \mbN_0 } 
    \bigg( 
        \ldots 
        \bigg( 
            \sum_{ n^{j_d} \in \mbN_0 } 
            a_{ n^1, \ldots, n^d } 
        \bigg)
    \bigg)    
\bigg) 
\\ 
& 
= 
\sum_{ n^{j_1} < N^1 } 
\sum_{ n^{j_2} < N^2 } 
\ldots 
\sum_{ n^{j_d} < N^d } 
a_{ n^1, \ldots, n^d }
= 
S_{\mathbf{N}},  
\end{split} 
\end{equation}
where $N^2$, $\ldots$, $N^d$ are taken sufficiently large so that for all nonzero terms of the series \eqref{Eq:WSer-Num} with $n^{j_1} < N^1$, the following inequalities hold:
\[ 
n^{j_2} < N^2, \quad \ldots, \quad n^{j_d} < N^d. 
\] 
Since by assumption the series \eqref{Eq:WSer-Num} 
converges to $S$ over rectangles, 
the right-hand side of \eqref{Eq:WSer-Rep-Conv-2}
tends to $S$ as $\min N^j \to \infty$. 
Consequently, the left-hand side also tends to $S$ as $N^1 \to \infty$, 
which means that equality \eqref{Eq:WSer-Rep-Conv-1} holds. 
The proposition is proved.
\end{proof}

\begin{lem}
If  $m_s \in \mathbb{N}_0$ and  
$\mathbf{m}$, $\mathbf{m}^{\prime}$, $\mathbf{p} < 2^{m_s} \mathbf{1}$, then 
\begin{equation} 
\label{Eq:WF-Scailing} 
W_{ 2^{m_s} \mathbf{p} } 
( \Delta_{ 2^{m_s} \mathbf{m} + \mathbf{m}^{\prime} }^{(2m_s)} ) 
= 
W_{ \mathbf{p} } 
( \Delta_{ \mathbf{m}^{\prime} }^{(m_s)} ) 
=: 
W_{ \mathbf{p} \mathbf{m}^{\prime} }^{(m_s)}. 
\end{equation} 
\end{lem}

\begin{proof} 
First, we prove \eqref{Eq:WF-Scailing} 
for the one-dimensional case.  
If 
\[ 
p 
= 
\sum\limits_{k < m_s} 
p_k 2^k  
=: ( p_{m_s -1} \, \ldots \, p_0 )_2, 
\quad 
m = ( a_{m_s -1} \, \ldots \, a_0 )_2, 
\quad  
m^{\prime} = ( b_{m_s -1} \, \ldots \, b_0 )_2  
\]
are the binary expansions of the numbers $p$, $m$ and $m^{\prime}$, 
then 
\[ 
2^{m_s} m + m^{\prime} 
= 
( a_{m_s -1} \, \ldots \, a_0 \, b_{m_s -1} \, \ldots \, b_0 )_2,  
\quad  
2^{m_s}p = ( p_{m_s -1} \, \ldots \, p_0 \, 
\underset{\text{$m_s$ zeros}}{ \underbrace{0 \, \ldots \, 0} } )_2,   
\] 
\begin{equation} 
\label{Eq:WF-Scailing-1} 
\begin{split}
W_{ 2^{m_s} p } 
( \Delta_{ 2^{m_s} m + m^{\prime} }^{(2m_s)} ) 
& 
\stackrel{\eqref{Eq:WF-DI-1}}{=} 
(-1)^{ a_{m_s -1} \cdot 0 + \ldots + a_0 \cdot 0 + b_{m_s -1} p_0 + \ldots + b_0 p_{m_s -1} } 
\\ 
& 
= 
(-1)^{ b_{m_s -1} p_0 + \ldots + b_0 p_{m_s -1} } 
= 
W_p ( \Delta_{ m^{\prime} }^{(m_s)} ).  
\end{split} 
\end{equation}
In the $d$-dimensional case, we obtain  
\[ 
\begin{split}
W_{ 2^{m_s} \mathbf{p} } 
( \Delta_{ 2^{m_s} \mathbf{m} + \mathbf{m}^{\prime} }^{(2m_s)} ) 
& 
= 
\prod\limits_{j=1}^d 
W_{ 2^{m_s} p^j } 
( \Delta_{ 2^{m_s} m^j + ( m^{\prime} )^j }^{(2m_s)} ) 
\\ 
& 
\stackrel{\eqref{Eq:WF-Scailing-1}}{=} 
\prod\limits_{j=1}^d 
W_{ p^j } 
( \Delta_{ ( m^{\prime} )^j }^{(m_s)} ) 
= 
W_{ \mathbf{p} } 
( \Delta_{ \mathbf{m}^{\prime} }^{(m_s)} ), 
\end{split}
\]
which completes the proof. 
\end{proof} 


\begin{lem}
If $N = 2^{k_1} + \ldots + 2^{k_s}$, 
$k_1 > \ldots > k_s$, then 
\begin{equation} 
\label{Eq:WFS-5} 
D_N 
= 
\sum\limits_{j=1}^s  
R_{k_1} \ldots R_{k_{j-1}} 
D_{2^{k_j}}. 
\end{equation} 
\end{lem}

\begin{proof} 
Using formula \eqref{Eq:WFS-2} several times, we obtain
\[ 
\begin{split} 
D_N 
& 
= 
D_{ 2^{k_1} + \ldots + 2^{k_s} } 
\\ 
& 
= 
D_{ 2^{k_1} } + R_{k_1} D_{ 2^{k_2} + \ldots + 2^{k_s} }  
\\
& 
= 
D_{ 2^{k_1} } 
+ 
R_{k_1} D_{ 2^{k_2} }
+ 
R_{k_1} R_{k_2}
D_{ 2^{k_3} + \ldots + 2^{k_s} } 
\\
&
\ldots\ldots\ldots\ldots\ldots\ldots\ldots\ldots\ldots\ldots\ldots\ldots
\\
& 
= 
D_{ 2^{k_1} } 
+ 
R_{k_1} D_{ 2^{k_2} }
+ 
R_{k_1} R_{k_2} D_{ 2^{k_3} }
+ 
\ldots 
+
R_{k_1} R_{k_2} \ldots R_{k_{s-1}} D_{ 2^{k_s} }.
\end{split} 
\] 
\end{proof}


\begin{prop} 
\label{Prop:Main-2}
Let a dyadic cube $\Delta^{(w)}$ $($of rank $w)$ and a series of the form \eqref{Eq:WSer} be given, and let $\tau$ be the quasi-measure generated by this series, as in \eqref{Eq:Canon-Iso}.
If $\tau (\Delta) = 0$ for any dyadic cube 
$\Delta \subset \Delta^{(w)}$, then  
$S_{2^{w} \mathbf{M}} ( \mathbf{g} ) = 0$  
for all $\mathbf{g} \in \Delta^{(w)}$ 
and $\mathbf{M} \in \mbN^d$.  
\end{prop} 

\begin{proof} 
Equation \eqref{Eq:PS-IR} gives 
\begin{equation} 
\label{Eq:WFS-12}
S_{ 2^{w} \mathbf{M} } ( \mathbf{g} ) 
=
\sum\limits_{ \mathbf{m} < 2^k \mathbf{1} } 
\tau 
\big( \mathbf{g} \oplus \Delta^{(k)}_{\mathbf{m}} \big) 
D_{ 2^{w} \mathbf{M} }  
\big( \Delta^{(k)}_{\mathbf{m}} \big), 
\end{equation}
where $k$ is chosen such that $2^{w} \mathbf{M} \le 2^k \mathbf{1}$. 

Each component of the vector $2^{w}\mathbf{M}$ is of the form
\[ 
2^{k_1} + \ldots + 2^{k_s}, 
\quad 
k_1 > \ldots > k_s \ge w,
\] 
with corresponding $k_j$. 
Therefore, from formula \eqref{Eq:WFS-5} 
combined with \eqref{Eq:WFS-11} it follows 
that the Dirichlet kernel $D_{2^{w}\mathbf{M}}(\mathbf{g})$ 
is a finite sum of Dirichlet kernels 
of the form $D_{2^{w^1},\ldots,2^{w^d}}(\mathbf{g})$, 
multiplied by values at the point $\mathbf{g}$ 
of some $d$-dim Walsh functions, 
where all $w^i \geq w$.
If  $\mathbf{g} \notin \Delta_{\mathbf{0}}^{(w)}$, then all $D_{2^{w^1}, \ldots, 2^{w^d}} ( \mathbf{g} )$ are zero by \eqref{Eq:WFS-3} and \eqref{Eq:WFS-11}, 
hence $D_{ 2^{w} \mathbf{M} } ( \mathbf{g} )$ also vanishes.  
Therefore, every term in the right-hand side of \eqref{Eq:WFS-12} 
satisfying $\Delta^{(k)}_{\mathbf{m}} \subsetneqq \Delta_{\mathbf{0}}^{(w)}$ 
is equal to zero.

If $\Delta^{(k)}_{\mathbf{m}} \subset \Delta_{\mathbf{0}}^{(w)}$, 
then $\mathbf{g} \oplus \Delta^{(k)}_{\mathbf{m}} \subset \Delta^{(w)}$, 
and consequently $\tau 
\big( \mathbf{g} \oplus \Delta^{(k)}_{\mathbf{m}} \big) = 0$ 
by the conditions of the proposition.
 
As a result, all terms on the right-hand side of \eqref{Eq:WFS-12} 
equal zero, and therefore the sum vanishes as well. 
This concludes the proof of the proposition.
\end{proof}


The next result establishes that, at every point outside the support 
of the quasi-measure generated by the multiple Walsh series, 
there exists a sufficiently massive subsequence of partial sums vanishing at that point.

\begin{prop} 
\label{Prop:Main-3}
Let a series of the form \eqref{Eq:WSer} be given,  
and denote by $\tau$ the quasi-measure generated by it via \eqref{Eq:Canon-Iso}.
If $\mathbf{g} \notin\mathrm{supp}\, \tau$,
there is a non-negative integer $w = w ( \mathbf{g} )$ this, 
that $S_{2^{w} \mathbf{M}} ( \mathbf{g} ) = 0$  
for all $\mathbf{M}\in\mbN^d$.

The value of $w$ may be chosen as the minimal rank 
of dyadic cubes $\Delta^{(w)}$ satisfying 
$\tau (\Delta) = 0$ for every dyadic cube $\Delta \subset \Delta^{(w)}$.
\end{prop} 

\begin{proof} 
By the definition of the quasi-measure's support,
for a given point $\mathbf{g}$ outside it,
there exists a dyadic cube $\Delta^{(w)} \ni \mathbf{g}$
$($of rank $w)$ such that $\tau(\Delta) = 0$ for any dyadic cube
$\Delta \subset \Delta^{(w)}$.
It remains to apply Proposition~\ref{Prop:Main-2}. 
\end{proof}


\begin{cor} 
\label{Cor:Main-4} 
Let $\tau = \tau_F$ be the quasi-measure constructed in Subsubs.~\ref{Subsubs:Tau-F} 
for a non-empte closed set $F$.
Then,  for every $\mathbf{g} \notin F$,  
there exists a non-negative integer $w = w(\mathbf{g})$ such that 
the partial sums of a series \eqref{Eq:WSer} 
that generates the quasi-measure $\tau$ 
satisfy the equality 
$S_{2^{w} \mathbf{M}} ( \mathbf{g} ) = 0$  
for all $\mathbf{M} \in \mbN^d$. 
\end{cor}


\section{Construction of $M$-set for multiple Walsh series and null-series}
\label{S:M-Set}

Consider the sequence of non-negative integers $(m_s, \, s \in \mathbb{N})$  
satisfying $m_1 = 0$ and the recurrence relation $m_{s+1} = 2(2m_s + 1)$.

For each $s$,  
we represent the group $\mathbb{G}^d$  
as a disjoint union of  
$2^{m_s d}$ dyadic cubes of rank $m_s$,  
and as a disjoint union of  
$2^{2m_s d}$ dyadic cubes of rank $2m_s$:
\[ 
\begin{split}
\mG^d 
& 
= 
\bigsqcup\limits_{ \mathbf{m} < 2^{m_s} \cdot \mathbf{1} } 
\Delta_{ \mathbf{m} }^{(m_s)} 
\\
& 
= 
\bigsqcup\limits_{ \mathbf{m} < 2^{m_s} \cdot \mathbf{1} } 
\bigsqcup\limits_{ \mathbf{m}^{\prime} < 2^{m_s} \cdot \mathbf{1} } 
\Delta_{ 2^{m_s} \mathbf{m} + \mathbf{m}^{\prime} }^{(2m_s)}. 
\end{split} 
\]  
Obviously, 
$\Delta_{ 2^{m_s} \mathbf{m} + \mathbf{m}^{\prime} }^{(2m_s)} 
\subset \Delta_{ \mathbf{m} }^{(m_s)}$ 
for all admissible $\mathbf{m}$ и $\mathbf{m}^{\prime}$. 
Let  
\begin{equation} 
\label{Eq:Tau-Coeff-001}  
F_s 
:= 
\bigsqcup\limits_{ \mathbf{m} < 2^{m_s} \cdot \mathbf{1} } 
\bigsqcup\limits_{ \mathbf{m}^{\prime} < 2^{m_s} \cdot \mathbf{1} } 
\big\{ 
    \mg \in \Delta_{ 2^{m_s} \mathbf{m} + \mathbf{m}^{\prime} }^{(2m_s)} 
    \colon 
    R_{ 2m_s \textbf{1} } (\mg) = W_{ \mathbf{m} \mathbf{m}^{\prime} }^{^{(m_s)}}  
\big\}.
\end{equation} 

If we decompose each dyadic cube $\Delta_{ 2^{m_s} \mathbf{m} + \mathbf{m}^{\prime} }^{(2m_s)}$ into $2^d$ contiguous dyadic cubes of rank $2m_s + 1$,  
we find that exactly $2^{d-1}$ of them lie within $F_s$,  
while the remaining $2^{d-1}$ are disjoint from $F_s$. 

We provide a geometric illustration of the set $F_s$.  
For each $\mathbf{m} < 2^{m_s} \mathbf{1}$,  
consider the $d$-dim Walsh function 
$W_{ 2^{m_s} \textbf{1} + \mathbf{m} } = R_{ m_s \textbf{1} } W_{ \mathbf{m} }$ 
with index from the binary block $B_{m_s} 
= 
\big\{ 
    2^{m_s} \mathbf{1} \le \mn < 2^{m_s + 1} \mathbf{1}   
\big\}$ 
and consider its ``scaled graph''   
or more precisely, the level set  
corresponding to the value $1$,  
compressed along each coordinate by a factor of $2^{m_s}$  
and placed within the dyadic cube $\Delta^{(m_s)}_{\mathbf{m}}$.
As a result, we get a set under the sign
of the inner union in \eqref{Eq:Tau-Coeff-001}, 
and the union by $\mathbf{m}$ of all such sets is $F_s$.
In Fig. 1, the set $F_2$ is shown in black.

\vspace{5pt}
\centerline{
\includegraphics[width=70mm, height=70mm]{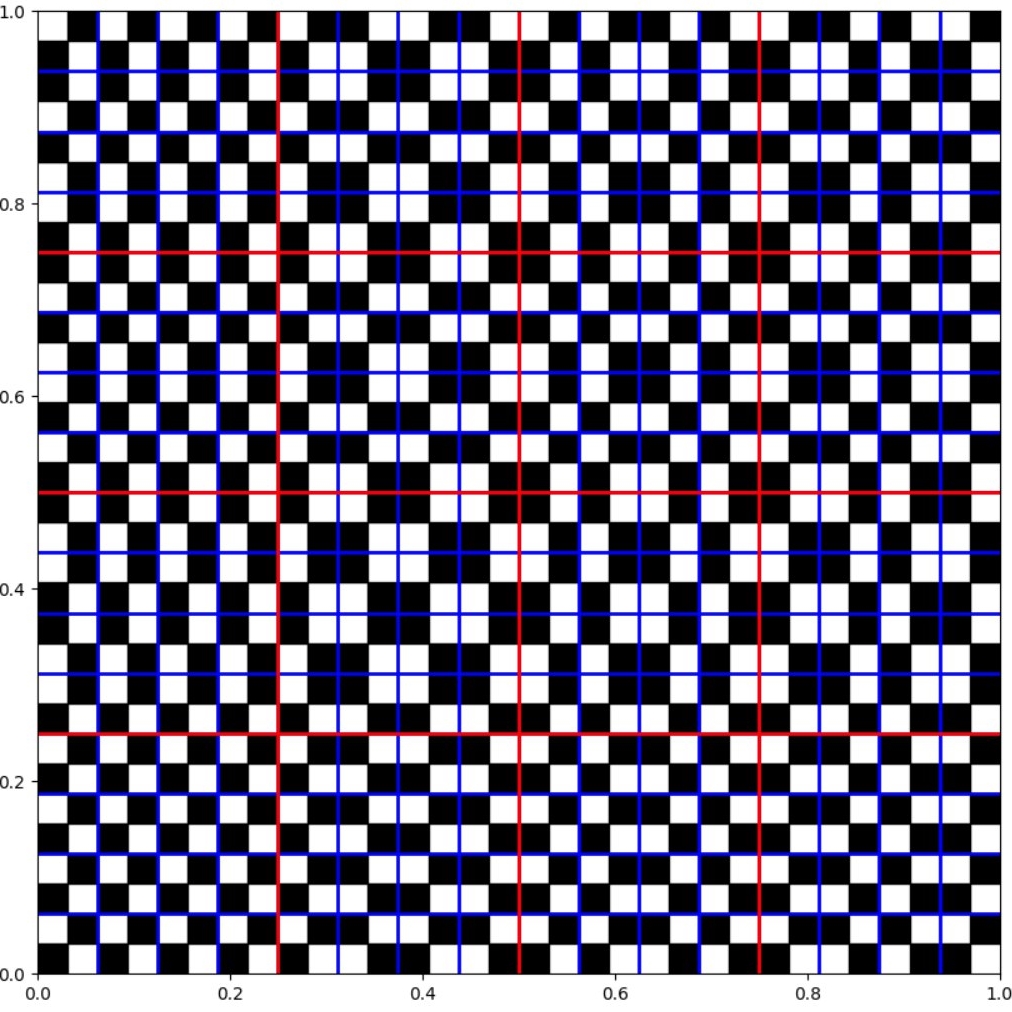}
}

Finally, let 
\begin{equation} 
\label{Eq:Tau-Coeff-002} 
F 
:= 
\bigcap\limits_{s=1}^\infty 
F_s, 
\qquad 
\widetilde{F}_s
:= 
\bigcap\limits_{k=1}^s 
F_k, 
\quad 
s \in \mathbb{N}. 
\end{equation}  

Alongside $F$, we consider the class of sets 
obtained from $F$ in the following way. 
For the sequence $(m_s)$ defined above, 
let $S$ denote the direct product of the symmetric groups $S( (0:2^{m_s}-1 )^d)$ 
of permutations of the set $( 0:2^{m_s}-1 )^d$.
Every element of $S$ is a sequence
$\boldsymbol{\pi} = ( \boldsymbol{\pi}_s, \, s \in \mathbb{N} )$, 
which generates sets 
\begin{equation}
\label{Eq:Fspi}
F_s^{\boldsymbol{\pi}} := 
\bigsqcup\limits_{ \mathbf{m} < 2^{m_s} \cdot \mathbf{1} } 
\bigsqcup\limits_{ \mathbf{m}^{\prime} < 2^{m_s} \cdot \mathbf{1} } 
\big\{ 
    \mg \in \Delta_{ 2^{m_s} \mathbf{m}+\mathbf{m}^{\prime} }^{(2m_s)} 
    \colon 
    R_{ 2m_s \textbf{1} } (\mg) = W_{ \mathbf{\boldsymbol{\sigma}_s(m)} \mathbf{m}^{\prime} }^{(m_s)}  
\big\}, 
\qquad 
s \in \mathbb{N}.
\end{equation}

In other words, the set $F_s^{\boldsymbol{\pi}}$ 
is obtained from $F_s$ by applying 
a transformation that shifts each portion  
$F_s \cap \Delta_{\mathbf{m}}^{(m_s)}$ 
by the set $\Delta_{\mathbf{m} \oplus \boldsymbol{\pi}_s(\mathbf{m})}^{(m_s)}$.  
Consequently, $F_s^{\boldsymbol{\pi}}$ is 
the union of shifted portions.

Let
\begin{equation}
\label{Eq:Fpi}
F^{\boldsymbol{\pi}} 
:= 
\bigcap_{s=1}^{\infty} 
F_s^{\boldsymbol{\pi}}, 
\qquad 
\widetilde{F}_s^{\boldsymbol{\pi}} 
:= 
\bigcap\limits_{k=1}^s 
F_s^{\boldsymbol{\pi}}, 
\quad 
s \in \mbN.  
\end{equation} 

Comparing formulas \eqref{Eq:Fspi} and \eqref{Eq:Fpi} 
with \eqref{Eq:Tau-Coeff-001} and \eqref{Eq:Tau-Coeff-002}, 
we observe that when all permutations 
$\boldsymbol{\pi}_s$ are identity maps, 
the sets $F_s^{\boldsymbol{\pi}}$, 
$\widetilde{F}_s^{\boldsymbol{\pi}}$ 
and $F^{\boldsymbol{\pi}}$ 
reduce to $F_s$, $\widetilde{F}_s$, and $F$, respectively.

It is straightforward to demonstrate $\mu \widetilde{F}_s^{\boldsymbol{\pi}} = 2^{-s}$ and $\mu F^{\boldsymbol{\pi}} = 0$. 

Consider the subgroup $S^{\prime} \subseteq S$ 
of elements $\boldsymbol{\pi} = ( \boldsymbol{\pi}_s, \, s \in \mathbb{N} )$ 
where $\boldsymbol{\pi}_s = \times_{j=1}^d \pi_s^j$, 
which means 
\begin{equation}
\label{Eq:pis-1} 
\boldsymbol{\pi}_s ( \mathbf{m} ) 
= 
( \pi_s^1 ( m^1), \ldots, \pi_s^d ( m^d ) ), 
\quad  
\mathbf{m} \in \{ 0 : 2 m_s - 1 \}^d. 
\end{equation}   

In \S~ \ref{S:Main-results} (Theorems~
\ref{T:Main-1} и \ref{T:Main-1-1}) we will prove that $F$ and all sets $F^{\boldsymbol{\pi}}$ with $\boldsymbol{\pi} \in S^{\prime}$ 
are $M$-sets for the $d$-dim Walsh system under convergence over rectangles
and, as a corollary, over cubes, as well as under iterated convergence. 
We now construct the null-series realizing these $M$-sets, while the proof that these are indeed null-series will be given in Theorems~\ref{T:Main-1} and \ref{T:Main-1-1} themselves. Let $\tau := \tau_F$ be the quasi-measure constructed from $F$ via the construction in Section~\ref{Subsubs:Tau-F}, 
and consider the multiple Walsh series 
$\sum\limits_{ \mn = \mathbf{0} }^\infty \tau_{\mn} W_{\mn} ( \mg )$ 
that generates it. 

Note that the coefficients $\tau_{\mathbf{n}}$  
of this series coincide  
with the Walsh--Fourier coefficients $\widehat{\tau}_{\mathbf{n}}$  
of the quasi-measure $\tau$, for which explicit formulas  
were established in Lemmas~\ref{Lem:2-1} and \ref{Lem:2-1-1}. 

Similarly, the quasi-measure $\tau^{\boldsymbol{\pi}} := \tau_{F^{\boldsymbol{\pi}}}$, 
is generated by some series 
$\sum\limits_{ \mn = \mathbf{0} }^\infty 
\tau^{\boldsymbol{\pi}}_{\mn} W_{\mn} ( \mg )$ 
with coefficients  
$( \tau^{\boldsymbol{\pi}} )_{\mn}  \equiv \widehat{\tau^{\boldsymbol{\pi}}}_{\mn}$,   
given in Remark~\ref{Rem:2-1} and Lemma~\ref{Lem:2-1-1}.  

From formulas \eqref{Eq:Tau-Coeff-001} and \eqref{Eq:Tau-Coeff-002} that define the set $F$, 
it follows that in formula \eqref{Eq:Tau-F-111}, 
describing the construction algorithm of the quasi-measure $\tau_F$, 
the constant $M$ equals  $2^{d-1}$, if $k = 2m_s$ 
for some $s$ and $2^d$ in other cases. 
Moreover, for dyadic cubes $\Delta^{(m_s)}$ 
relations $\Delta^{(m_s)} \cap F \neq \emptyset$ 
and $\Delta^{(m_s)} \subset \widetilde{F}_{s-1}$ 
are equivalent, 
while for  $\Delta^{(2m_s +1)}$ relations $\Delta^{(2m_s+1)} \cap F \neq \emptyset$ 
and $\Delta^{(2m_s+1)} \subset F_s$  are equivalent. 
This yields: 
\begin{equation} 
\label{Eq:Tau-Coeff-01} 
\tau( \Delta^{(m_s)} ) 
= 
\mathrm{I}
( \Delta^{(m_s)} \subset \widetilde{F}_{s-1} ) 
\frac{2^s}{2^{dm_s+1}};  
\end{equation} 
\begin{equation} 
\label{Eq:Tau-Coeff-03} 
\begin{split}
\tau( \Delta^{( 2m_s+1 )} ) 
& 
= 
\mathrm{I}
( \Delta^{(2m_s+1)} \subset F_s ) 
\, 2^{s - d (2m_s+1)} 
\\
& 
= 
2^s 
\mathrm{I}
( \Delta^{(2m_s +1)} \subset F_s ) 
\,
\mu( \Delta^{( 2m_s + 1 )} ). 
\end{split} 
\end{equation}

The formulas remain valid 
if we substitute 
$\tau$, $\widetilde{F}_{s-1}$ and $F_s$ 
with $\tau^{\boldsymbol{\pi}}$, $\widetilde{F}_{s-1}^{\boldsymbol{\pi}}$, and $F_s^{\boldsymbol{\pi}}$ 
respectively. 

If we restrict the quasi-measure $\tau$ or $\tau^{\boldsymbol{\pi}}$ to a dyadic cube $\Delta^{(m_{s_0})}$ that has nonempty intersection with $F$, then for $s \ge s_0$, the values of the obtained 
quasi-measure on $\Delta^{(2m_s+1)}$ 
and $\Delta^{(2m_s+1)}$ are determined  
by formulas \eqref{Eq:Tau-Coeff-01} and \eqref{Eq:Tau-Coeff-03} 
if the sets $\widetilde{F}_{s-1}$ and $F_s$ are replaced 
by their intersections with $\Delta^{(m_{s_0})}$.


\section{Auxliary lemmas}
\label{S:Prelim-3}


For the quasi-measure $\tau$ and all dyadic cubes $\Delta = \Delta^{(m(s))}$, 
we consider the quantity
\[ 
\widehat{\tau}_{\mn} ( \Delta ) 
:= 
\int\limits_{\Delta} 
W_{\mn} \, d \tau   
\]  
(in \cite{skvortsov-1977-1} the quantities $\widehat{\tau}_{\mn} ( \Delta ) / \mu \Delta$ 
are called the \emph{local Walsh--Fourier coefficients} 
of a quasi-measure $\tau$). 
Clearly,  
\begin{equation} 
\label{Eq:Loc-FC} 
\widehat{\tau}_{\mn}  
= 
\sum\limits_{ \mathbf{m} < 2^k \mathbf{1} }
\widehat{\tau}_{\mn} ( \Delta_{\mathbf{m}}^{(k)} ). 
\end{equation}


\begin{lem} 
\label{Lem:2-1}
Suppose that  
\begin{equation} 
\label{Eq:2}
\mn 
= 
2^{2m_s} \mathbf{1} 
+ 
2^{m_s} \mathbf{p} 
+ 
\mathbf{q}, 
\qquad 
\mathbf{0} 
\le 
\mathbf{p}, \, \mathbf{q} 
< 
2^{m_s} \mathbf{1}.  
\end{equation} 

Then for the quasi-measure $\tau = \tau_F$ constructed in \S~1, the following holds$:$
\begin{equation} 
\label{Eq:Tau-Coeff-000} 
\widehat{\tau}_{\mn} ( \Delta^{(m_s)}_{\mathbf{m}} )  
= 
2^{s - 1 - d m_s}
W_{ \mathbf{q} \mathbf{m} }^{(m_s)} 
\de_{\mathbf{m}}^{\mathbf{p}} 
\mathrm{I} ( \Delta_{ \mathbf{m} }^{(m_s)} \subset \widetilde{F}_{s-1} ); 
\end{equation} 
\begin{equation} 
\label{Eq:Tau-Coeff-0} 
\widehat{\tau}_{\mn} 
= 
2^{s - 1 - d m_s} 
W_{\mathbf{q} \mathbf{p}}^{(m_s)}  
\, 
\mathrm{I}( \Delta^{(m_s)}_{\mathbf{p}} 
\subset \widetilde{F}_{s-1} ).
\end{equation}
\end{lem} 

\begin{proof} 
Taking into account equality \eqref{Eq:WF-Scailing}, we have
\begin{equation} 
\label{Eq:Tau-Coeff-02}  
\begin{split}
\widehat{\tau}_{\mn} 
( \Delta^{(m_s)}_{\mathbf{m}} )
& 
= 
\int\limits_{ \Delta^{(m_s)}_{\mathbf{m}} } 
W_{\mn} \1 d \tau  
= 
\int\limits_{ \Delta^{(m_s)}_{\mathbf{m}} } 
R_{ 2m_s \textbf{1} }  
W_{ 2^{m_s} \mathbf{p} }  
W_{\mathbf{q}} \, d \tau  
\\ 
& 
= 
W_{ \mathbf{q} \mathbf{m} }^{(m_s)}
\sum\limits_{ \mathbf{m}^{\prime} < 2^{m_s} \mathbf{1} } 
\int\limits_{ \Delta_{ 2^{m_s} \mathbf{m} + \mathbf{m}^{\prime} }^{(2m_s)} } 
R_{ 2m_s \textbf{1} } W_{ 2^{m_s} \mathbf{p} } \, d \tau  
\\ 
& 
= 
W_{ \mathbf{q} \mathbf{m} }^{(m_s)}
\sum\limits_{ \mathbf{m}^{\prime} < 2^{m_s} \mathbf{1} } 
W_{ \mathbf{p} \mathbf{m}^{\prime} }^{(m_s)} 
\sum\limits_{ \boldsymbol{\sigma} \in \{ 0, 1 \}^d }
S_{\boldsymbol{\sigma}}, 
\\ 
S_{\boldsymbol{\sigma}}
& 
:= 
R_{ 2m_s \textbf{1} } 
( \Delta_{ 2^{m_s+1} \mathbf{m} + 2\mathbf{m}^{\prime} + \boldsymbol{\sigma} }^{(2m_s+1)} ) 
\tau 
( \Delta_{ 2^{m_s+1} \mathbf{m} + 2\mathbf{m}^{\prime} + \boldsymbol{\sigma} }^{(2m_s+1)} ).  
\end{split}
\end{equation}

If $\Delta^{(m_s)}_{\mathbf{m}}$ does not belong to $\widetilde{F}_{s-1}$, 
then all $\Delta_{ 2^{m_s+1} \mathbf{m} + 2\mathbf{m}^{\prime} + \boldsymbol{\sigma} }^{(2m_s+1)}$ 
also do not belong, and therefore all
$\tau ( \Delta_{ 2^{m_s+1} \mathbf{m} + 2\mathbf{m}^{\prime} + \boldsymbol{\sigma} }^{(2m_s+1)} ) = 0$ 
and $S_{\boldsymbol{\sigma}} = 0$. 
If, however,  $\Delta^{(m_s)}_{\mathbf{m}} \subset \widetilde{F}_{s-1}$, 
then 
\[ 
\tau 
( \Delta_{ 2^{m_s+1} \mathbf{m} + 2\mathbf{m}^{\prime} + \boldsymbol{\sigma} }^{(2m_s+1)} ) 
= 0 
\] 
for half of the vectors $\boldsymbol{\sigma} \in \{ 0, 1 \}^d$, and
\begin{equation} 
\label{Eq:Tau-Coeff-02-2} 
R_{ 2m_s \textbf{1} } 
( \Delta_{ 2^{m_s+1} \mathbf{m} + 2\mathbf{m}^{\prime} + \boldsymbol{\sigma} }^{(2m_s+1)} ) 
\stackrel{\eqref{Eq:Tau-Coeff-001}}{=}
W_{ \mathbf{m} \mathbf{m}^{\prime} }^{(m_s)} 
\end{equation}  
and
\[   
\tau 
( \Delta_{ 2^{m_s+1} \mathbf{m} + 2\mathbf{m}^{\prime} + \boldsymbol{\sigma} }^{(2m_s+1)} )= 
2^{s - d (2m_s+1)}      
\]  
--- for the other half. 
In view of the above,
\begin{equation} 
\label{Eq:Tau-Coeff-12}  
\begin{split}
\widehat{\tau}_{\mn} 
( \Delta^{(m_s)}_{\mathbf{m}} ) 
& 
= 
2^{s- d (2m_s+1)} 2^{d-1}
W_{ \mathbf{q} \mathbf{m} }^{(m_s)}
\sum\limits_{ \mathbf{m}^{\prime} < 2^{m_s} \mathbf{1} }
W_{ \mathbf{p} \mathbf{m}^{\prime} }^{(m_s)}
W_{ \mathbf{m} \mathbf{m}^{\prime} }^{(m_s)}  
\\ 
& 
= 
2^{s - 1 - 2 d m_s} 
W_{ \mathbf{q} \mathbf{m} }^{(m_s)} 
\sum\limits_{ \mathbf{m}^{\prime} < 2^{m_s} \mathbf{1} }
W_{ \mathbf{p} \mathbf{m}^{\prime} }^{(m_s)} 
W_{ \mathbf{m} \mathbf{m}^{\prime} }^{(m_s)}  
\\ 
& 
\stackrel{\eqref{Eq:Tau-Coeff-03}}{=}
2^{s - 1 - 2 d m_s} 
W_{ \mathbf{q} \mathbf{m} }^{(m_s)} 
2^{dm_s} 
\de_{\mathbf{m}}^{\mathbf{p}} 
\stackrel{\eqref{Eq:Tau-Coeff-03}}{=}
2^{s - 1 - d m_s} 
W_{ \mathbf{q} \mathbf{m} }^{(m_s)} 
\de_{\mathbf{m}}^{\mathbf{p}}.
\end{split} 
\end{equation} 
From the above reasoning, we obtain formula \eqref{Eq:Tau-Coeff-001} 
that implies \eqref{Eq:Tau-Coeff-01}: 
\begin{equation} 
\label{Eq:Tau-Coeff-14}  
\begin{split}
\widehat{\tau}_{\mn}  
& 
\stackrel{\eqref{Eq:Loc-FC}}{=}
\sum\limits_{ \mathbf{m} < 2^{m_s} \mathbf{1} } 
\widehat{\tau}_{\mn} ( \Delta^{(m_s)}_{\mathbf{m}} ) 
\\ 
& 
\stackrel{\eqref{Eq:Tau-Coeff-12}}{=}
2^{s - 1 - d m_s} 
\sum\limits_{ \mathbf{m} < 2^{m_s} \mathbf{1} } 
W_{ \mathbf{q} \mathbf{m} }^{(m_s)} 
\de_{\mathbf{m}}^{\mathbf{p}}
= 
2^{s - 1 - d m_s} 
W_{ \mathbf{q} \mathbf{p} }^{(m_s)}. 
\end{split} 
\end{equation} 
The lemma is proved.
\end{proof} 


\begin{rem} 
\label{Rem:2-1} Let $\boldsymbol{\pi} \in S$.
If the vector $\mathbf{n}$ has the form \eqref{Eq:2}, 
then 
\begin{equation} 
\label{Eq:Tau-Coeff-15} 
( \widehat{ \tau^{\boldsymbol{\pi}} } )_{\mn} ( \Delta^{(m_s)} )_{\mathbf{m}}   
= 
2^{s - 1 - d m_s}
W_{ \mathbf{q} \mathbf{m} }^{(m_s)} 
\de_{ \boldsymbol{\pi}_s ( \mathbf{m} ) }^{\mathbf{p}} 
\mathrm{I} ( \Delta_{\mathbf{m}}^{(m_s)} 
\subset \widetilde{F}_{s-1} ),  
\end{equation}  
the Fourier--Walsh coefficients of the quasi-measure $\tau^{\boldsymbol{\pi}}$ are given by  
\begin{equation} 
\label{Eq:Tau-Coeff-16} 
(\widehat{\tau^{\boldsymbol{\pi}}})_{\mn} 
= 
2^{s - 1 - d m_s} 
W_{\mathbf{q} \mathbf{\boldsymbol{\pi}^{-1}_s(p)} }^{(m_s)}  
\, 
\mathrm{I}( \Delta^{(m_s)}_{ \mathbf{\boldsymbol{\pi}^{-1}_s(p)} } 
\subset \widetilde{F}_{s-1} ).
\end{equation} 

Indeed, let us repeat the proof of Lemma~\ref{Lem:2-1}, 
replacing $\tau$ with $\tau^{\boldsymbol{\pi}}$ 
throughout, and 
$W_{ \mathbf{m} \mathbf{m}^{\prime} }^{(m_s)}$ 
with $W_{ \boldsymbol{\pi}^{-1}_s ( \mathbf{m} ) \mathbf{m}^{\prime} }^{(m_s)}$ 
in formulas \eqref{Eq:Tau-Coeff-02-2} and \eqref{Eq:Tau-Coeff-12}. 
This will yield formula \eqref{Eq:Tau-Coeff-15}, 
from which, via a chain of inequalities analogous to \eqref{Eq:Tau-Coeff-14} 
we obtain formula \eqref{Eq:Tau-Coeff-16}.     
\end{rem}


\begin{lem} 
\label{Lem:2-1-1} 
Suppose that the vector $\mathbf{n}$ does not belong to any of the binary blocks $B_{2m_s}$. 
Then $\widehat{\tau}_{\mathbf{0}} = 1$ 
and $\widehat{\tau}_{\mathbf{n}} = 0$ for $\mathbf{n} \neq \mathbf{0}$. 

Similar equalities hold for  
the Fourier--Walsh coefficients of all quasi-measures $\tau^{\boldsymbol{\pi}}$, 
$\boldsymbol{\pi} \in S$.  
\end{lem} 

\begin{proof} 
We have $\widehat{\tau}_{\mathbf{0}} 
= W_{\mathbf{0}} ( \Delta^{(0)}_{\mathbf{0}} ) 
\tau ( \Delta^{(0)}_{\mathbf{0}} ) = 1 \cdot 1 = 1$.  

Furthermore, let $\mathbf{n} \neq \mathbf{0}$ 
and assume the vector $\mn \notin \bigcup_{s \in \mathbb{N}} B_{2m_s}$. 
Let $k \in \mathbb{N}$ be the integer satisfying $\max \{ n^1, \ldots, n^d \} \in [ 2^{ k }, 2^{k+1} )$, then 
\begin{equation}
\label{Eq:Rec-Q-M-30} 
\begin{split}
\widehat{\tau}_{\mn} 
& 
\stackrel{\eqref{Eq:Rec-Q-M}}{=} 
\sum\limits_{ \mathbf{m} < 2^{k+1} \mathbf{1} } 
W_{ \mathbf{n} }  
\big( \Delta^{(k+1)}_{\mathbf{m}} \big) 
\tau 
\big( \Delta^{(k+1)}_{\mathbf{m}} \big)
\\ 
&
= 
\sum\limits_{ \mathbf{m} < 2^k \mathbf{1} } 
\sum\limits_{ \boldsymbol{\sigma} \in \{ 0, 1 \}^d }  
W_{ \mathbf{n} }  
\big( \Delta^{(k+1)}_{ 2 \mathbf{m} + \boldsymbol{\sigma} } \big) 
\tau 
\big( \Delta^{(k+1)}_{ 2 \mathbf{m} + \boldsymbol{\sigma} } \big).  
\end{split} 
\end{equation}
Fix $\mathbf{m}$ and compute the inner sum $S_{\mathbf{m}}$  
on the right-hand side of \eqref{Eq:Rec-Q-M-30}.

Consider three possible cases.  

In the first case $\Delta^{(k)}_{\mathbf{m}}$ does not intersect  
$F$, and consequently all $\Delta^{(k+1)}_{ 2 \mathbf{m} + \boldsymbol{\sigma} }$ don't either.
Therefore, in this case,  
$\tau \big( \Delta^{(k+1)}_{ 2 \mathbf{m} + \boldsymbol{\sigma} } \big) = 0$ 
for all $\boldsymbol{\sigma}$ и $S_{\mathbf{m}} = 0$.   

In the second case
$\Delta^{(k+1)}_{ 2 \mathbf{m} + \boldsymbol{\sigma} } \cap F \ne \emptyset$ 
for all $\boldsymbol{\sigma}$. 
Then all values 
$\tau \big( \Delta^{(k+1)}_{ 2 \mathbf{m} + \boldsymbol{\sigma} } \big)$,  
$\boldsymbol{\sigma} \in \{ 0, 1 \}^d$, 
are identical. 
The properties of Walsh functions imply that the values  
$W_{\mathbf{n}}\big(\Delta^{(k+1)}_{2\mathbf{m} + \boldsymbol{\sigma}}\big)$  
take $\pm 1$ equally often  
as $\boldsymbol{\sigma}$ ranges over $\{0,1\}^d$.
And in this case $S_{\mathbf{m}} = 0$.

In the third case, $\Delta^{(k+1)}_{ 2 \mathbf{m} + \boldsymbol{\sigma} } \cap F \ne \emptyset$  holds  
for some $\boldsymbol{\sigma} \in \{0,1\}^d$, but not for all.  
This is only possible when $k = 2m(s)$ for some $s$.  
Since the vector $\mathbf{n}$ does not have the form \eqref{Eq:2},  
there exist indices $n^l \notin [2^{2m(s)}, 2^{2m(s)+1})$  
for all $l$ in some index set $A \subset 1:d$  
with cardinality $L \geq 1$. 
We partition the terms of the inner sum on the right-hand side of \eqref{Eq:Rec-Q-M-30}  
into $2^{d-L}$ blocks of $2^L$ terms each,  
such that the vectors $\boldsymbol{\sigma}$ corresponding to  
terms within the same block differ only  
in coordinates indexed by $l \in A$.

Then the values of $W_{\mathbf{n}}\big(\Delta^{(k(s)+1)}_{ 2 \mathbf{m} + \boldsymbol{\sigma}} \big)$
are identical across all terms from the same block,
and for different blocks they take the values $\pm 1$  
the same number of times.

Furthermore, $2^{d-1}$ of the values  
$\tau ( \Delta^{( k(s)+1)}_{ 2 \mathbf{m} + \boldsymbol{\sigma} } )$ equal zero, while the remaining equal some constant $a$,  
with each block containing an same number of values $a$.
 
Hence we get that in this case $S_{\mathbf{m}} = 0$ and 
therefore $\widehat{\tau}_{\mathbf{n}} = 0$.

The preceding arguments remain valid when replacing $\tau$ with $\tau^{\boldsymbol{\pi}}$. 
The proof of the lemma is now complete.
\end{proof}


\begin{lem} 
\label{Lem:2-2}
If $\mn < 2^{2m_s+1} \mathbf{1}$, 
but $\mn \notin B_{2m_s}$ and $\mn \notin ( 0 : 2^{m_s} - 1 )^d$, 
then $\widehat{\tau}_{\mn} ( \Delta^{(m_s)}_{\mathbf{l}} ) = 0$ 
for all $\mathbf{l}$. 

The analogous equality holds for the Fourier--Walsh coefficients 
of all quasi-measures $\tau^{\boldsymbol{\pi}}$, 
$\boldsymbol{\pi} \in S$.  
\end{lem} 

\begin{proof} 
Let us take the natural number $k$ such that $\max \{ n^1, \ldots, n^d \} \in [ 2^{ k }, 2^{k+1} )$ 
and consider the following formula analogous to \eqref{Eq:Rec-Q-M-30}:
\[
\widehat{\tau}_{\mn} ( \Delta^{(m_s)}_{\mathbf{l}} ) 
= 
\sum\limits 
\sum\limits_{ \boldsymbol{\sigma} \in \{ 0, 1 \}^d }  
W_{ \mathbf{n} }  
\big( \Delta^{(k+1)}_{ 2 \mathbf{m} + \boldsymbol{\sigma} } \big) 
\tau 
\big( \Delta^{(k+1)}_{ 2 \mathbf{m} + \boldsymbol{\sigma} } \big),   
\] 
the outer summation is taken over all $\mathbf{m} < 2^k \mathbf{1}$ 
satisfying $\Delta^{(k)}_{\mathbf{m}} \subset \Delta^{(m_s)}_{\mathbf{l}}$;  
this summation is well-defined because
$\mn \notin ( 0 : 2^{m_s} - 1 )^d$ and therefore $k \ge m_s$. 
The remainder of the proof follows from exactly the same reasoning as in the proof of Lemma~\ref{Lem:2-1-1}. 
Thus, the lemma is proved.
\end{proof}


\begin{lem} 
\label{Lem:2-3}
If $\mn < 2^{2m_s+1} \mathbf{1}$ 
and $\Delta^{(m_s)} \subset \widetilde{F}_{s-1}$, 
then 
\[ 
\begin{split}
\int\limits_{ F_s \cap \Delta^{(m_s)} } W_{\mn} \, d \mu 
& 
= 
2^{-s} 
\int\limits_{ \Delta^{(m_s)} } W_{\mn} \, d \tau 
\\ 
& 
= 
2^{-s} 
\widehat{\tau}_{\mn} ( \Delta^{(m_s)} ). 
\end{split}
\] 
The statement remains valid when replacing $\widetilde{F}_{s-1}$ 
with $\widetilde{F}^{\boldsymbol{\pi}}_{s-1}$
and $\tau$ with $\tau^{\boldsymbol{\pi}}$ for any permutation 
$\boldsymbol{\pi} \in S$.  
\end{lem} 

\begin{proof} 
We have:
\[ 
\begin{split}
\int\limits_{ F_s \cap \Delta^{(m_s)} } W_{\mn} \, d \mu 
& 
= 
\sum\limits_{ \Delta^{(2m_s+1)} \subset F_s \cap \Delta^{(m_s)} } 
W_{\mn} ( \Delta^{(2m_s+1)} ) \mu ( \Delta^{(2m_s+1)} ) 
\\ 
& 
\stackrel{\eqref{Eq:Tau-Coeff-03}}{=} 
2^{-s} 
\sum\limits_{ \Delta^{(2m_s+1)} \subset F_s \cap \Delta^{(m_s)} } 
W_{\mn} ( \Delta^{(2m_s+1)} ) \tau ( \Delta^{(2m_s+1)} )
\\ 
& 
= 
2^{-s}
\sum\limits_{ \Delta^{(2m_s+1)} \subset F_s \cap \Delta^{(m_s)} } 
\int\limits_{ F_s \cap \Delta^{(m_s)} } W_{\mn} \, d \tau 
\\ 
& 
\stackrel{\text{?}}{=} 
2^{-s}
\int\limits_{ \Delta^{(m_s)} } W_{\mn} \, d \tau 
= 
2^{-s}
\widehat{\tau}_{\mn} ( \Delta^{(m_s)} ). 
\end{split} 
\] 
In the transition marked by ``?'', 
we used the fact that 
$\tau(\Delta) = 0$ 
for all dyadic cubes $\Delta \subset \mG^d \setminus F_s$. 

If we replace $\widetilde{F}_{s-1}$ with $\widetilde{F}^{\boldsymbol{\pi}}_{s-1}$ and $\tau$ with $\tau^{\boldsymbol{\pi}}$ for $\boldsymbol{\pi} \in S$, the proof is carried out similarly.
\end{proof} 


\begin{rem} 
\label{Rem:2-3} 
Lemmas~\ref{Lem:2-1}--\ref{Lem:2-3} and Remark~\ref{Rem:2-1} 
can be adapted, with minor modifications,  
to the case where instead of quasi-measures $\tau = \tau_F$ 
and $\tau^{\boldsymbol{\pi}} = \tau_{F^{\boldsymbol{\pi}}}$, 
we consider their restrictions 
$\psi := \tau|_{\Delta^{m_{s_0}}}$ and $\psi^{\boldsymbol{\pi}} := \tau^{\boldsymbol{\pi}}|_{\Delta^{m_{s_0}}}$ 
to a fixed dyadic cube $\Delta^{m_{s_0}}$. Taking into account the remarks after equations \eqref{Eq:Tau-Coeff-01} and \eqref{Eq:Tau-Coeff-03}, we obtain the following.

If the vector $\mathbf{n}$ has the form \eqref{Eq:2}, 
where $s > s_0$ (this condition ensures 
that all dyadic cubes appearing in the proof 
remain within $\Delta^{m_{s_0}}$), then 
\begin{equation} 
\label{Eq:Tau-Coeff-000-mod} 
\widehat{\psi}_{\mn} ( \Delta^{(m_s)}_{\mathbf{m}} )  
= 
2^{s - 1 - d m_s}
W_{ \mathbf{q} \mathbf{m} }^{(m_s)} 
\de_{\mathbf{m}}^{\mathbf{p}} 
\mathrm{I} ( \Delta_{ \mathbf{m} }^{(m_s)} \subset \widetilde{F}_{s-1} 
\cap \Delta^{m_{s_0}}),  
\end{equation} 

\begin{equation} 
\label{Eq:Tau-Coeff-0-mod} 
\widehat{\psi}_{\mn} 
= 
2^{s - 1 - d m_s} 
W_{\mathbf{q} \mathbf{p}}^{(m_s)}  
\, 
\mathrm{I}( \Delta^{(m_s)}_{\mathbf{p}} 
\subset 
\widetilde{F}_{s-1} \cap \Delta^{m_{s_0}} ) 
\end{equation} 
(compared to the right-hand sides of \eqref{Eq:Tau-Coeff-000} and \eqref{Eq:Tau-Coeff-0}, 
only the intersection with $\Delta^{m_{s_0}}$ appeared).  Taking into account Remark~\ref{Rem:2-1}, 
similar formulas can be derived for 
the quasi-measures $\psi^{\boldsymbol{\pi}}$.

Furthermore, for the restricted quasi-measures, the statements of Lemmas~\ref{Lem:2-1-1} and~\ref{Lem:2-2} remain valid provided that $\max \{ n^1, \ldots, n^d \} > 2^{2m_{s_0}}$ 
и $s > s_0$. 
Under the same conditions, the statement of Lemma~\ref{Lem:2-3} remains valid if we additionally consider the intersections of the sets $\widetilde{F}_{s-1}$ and $\widetilde{F}^{\boldsymbol{\pi}}_{s-1}$ with $\Delta^{m_{s_0}}$ instead of the original sets. 
\end{rem}


\section{Main results}
\label{S:Main-results}


\begin{theorem} 
\label{T:Main-1}
The set $F$ defined by \eqref{Eq:Tau-Coeff-001} and \eqref{Eq:Tau-Coeff-002},
is an $M$-set for the $d$-dim Walsh system under convergence over rectangles, 
cubes, or under iterated convergence for any order of summation. 
In this case, the $d$-multiple Walsh series 
$\sum_{\mathbf{n} = \mathbf{0}}^\infty \tau_{\mathbf{n}} W_{\mathbf{n}}(\mathbf{g})$ 
generating the quasi-measure $\tau := \tau_F$ constructed from the set $F$ 
in item.~\ref{Subsubs:Tau-F},
is a null-series that realizing the $M$-set $F$.
\end{theorem}

\begin{proof}
Let us immediately note that the coefficients $\tau_{\mathbf{n}}$ 
of this Walsh series coincide with 
the Walsh--Fourier coefficients $\widehat{\tau}_{\mathbf{n}}$ 
of the quasi-measure $\tau$. 

Fix an arbitrary point $\mg \in \mG^d \setminus F$. 
We need to prove that  
\begin{equation}
\label{Eq:WSer-PS-2} 
\lim S_{\mN} ( \mathbf{g} ) = 0 
\;\; \text{as $\min \{ N^1, \ldots, N^d \} \to \infty$}, 
\end{equation}
Since the indices of all nonzero coefficients $\tau_{\mathbf{n}}$  
lie in one of the diagonal binary blocks 
$B_{2 m_s} 
:= 
\big\{ 
    2^{2 m_s} \mathbf{1} \le \mn < 2^{2 m_s + 1} \mathbf{1}   
\big\}$,  
it suffices to consider in \eqref{Eq:WSer-PS-2} 
only vectors $\mathbf{N}$ belonging to $B_{2^{m_s}}$.
Corollary~\ref{Cor:Main-4} guarantees  
the existence of a natural number  
$s_0$ such that
\begin{equation}
\label{Eq:WSer-PS-3} 
S_{2^{m_s} \mathbf{M}} ( \mathbf{g} ) 
=
S_{2^{\frac{m_s}{2}} \mathbf{M}} ( \mathbf{g} ) 
= 0, 
\;\; \text{if $s \ge s_0$ и $\mathbf{M} \in \mbN^d$}   
\end{equation} 
(recall that $m_s/2 = 2m_{s-1}+1$). 

For $s \geq s_0$, we write the vector $\mathbf{N} \in B_{2^{m_s}}$ in the form 
\[ 
\mN 
= 
2^{2m_s} \mathbf{1}
+ 
2^{m_s} \mathbf{M} 
+ 
2^{\frac{m_s}{2}} \mathbf{L} 
+
\mathbf{K}, 
\quad 
\mathbf{M} < 2^{m_s} \mathbf{1}, 
\quad 
\mathbf{L} < 2^{\frac{m_s}{2}} \mathbf{1}, 
\quad 
\mathbf{K} < 2^{\frac{m_s}{2}} \mathbf{1}.  
\]
We have 
\begin{equation} 
\label{Eq:E}
S_{\mathbf{N}} ( \mathbf{g} ) 
= 
S_{\mathbf{N}^\prime} ( \mathbf{g} ) 
+ 
E 
\stackrel{\eqref{Eq:WSer-PS-3}}{=} 
E, 
\qquad
E 
:= 
S_{\mathbf{N}} ( \mathbf{g} ) 
- 
S_{\mathbf{N}^\prime} ( \mathbf{g} ),      
\end{equation}
where $\mathbf{N}^\prime := 2^{2m_s} \mathbf{1}
+ 2^{m_s} \mathbf{M} + 2^{\frac{m_s}{2}} \mathbf{L}$. 

Let us estimate $|E|$ from above. 
The quantity $E$ can be expressed as follows:
\begin{equation}
\label{Eq:WSer-PS-4} 
E 
= 
\sum\limits_{ \boldsymbol{\sigma} \in \{ 0, 1 \}^d \setminus \{ \mathbf{0}, \mathbf{1} \} }
\sum\limits_{ \mathbf{n} \in W_{\boldsymbol{\sigma}} }
\tau_{\mn} W_{\mn} ( \mg ) + \sum\limits_{ \mathbf{n} \in V }
\tau_{\mn} W_{\mn} ( \mg ).  
\end{equation} 
Here $W_{\boldsymbol{\sigma}}$ consists of all $\mathbf{n} \in B_{2^{m_s}}$ such that
\[ 
\begin{split}
2^{2m_s} + 2^{m_s} M^j + 2^{\frac{m_s}{2}} L^j 
\le 
n^j 
& 
< 
N^j  
\quad 
\text{if $\sigma^j = 1$},  
\\ 
2^{2m_s} 
\le
n^j 
& 
< 
2^{2m_s} + 2^{m_s} M^j 
\quad 
\text{if $\sigma^j = 0$},  
\end{split} 
\] 
while $V =  W_{\mathbf{1}} 
\bigsqcup 
\bigsqcup\limits_{\boldsymbol{\sigma} \in \{ 0,1 \}^d \setminus \{ \mathbf{0} \} } \widetilde{W}_{\boldsymbol{\sigma}}$, 
where $\widetilde{W}_{\boldsymbol{\sigma}}$ consists of all $\mathbf{n} \in B_{2 m_s}$ such that 
\[ 
\begin{split}
2^{2m_s} + 2^{m_s} M^j + 2^{\frac{m_s}{2}} L^j
\le 
n^j 
& 
< 
N^j
\quad 
\text{if $\sigma^j = 1$},  
\\ 
2^{2m_s} + 2^{m_s} M^j 
\le 
n^j 
& 
< 2^{2m_s} + 2^{m_s} M^j + 2^{\frac{m_s}{2}} L^j  
\quad 
\text{if $\sigma^j = 0$}.   
\end{split}
\] 

Let us estimate the cardinality of the set $V$: 
\begin{equation}
\label{Eq:WSer-PS-5-1}  
\begin{split} 
\# V 
= 
\# W_{\mathbf{1}} 
+ 
\sum\limits_{ \boldsymbol{\sigma} \in \{ 0,1 \}^d \setminus \{ \mathbf{0}, \mathbf{1} \} } 
\# \widetilde{W}_{\boldsymbol{\sigma}} 
& 
\le 
2^{d \frac{m_s}{2}} 
+ 
( 2^d - 2 ) 2^{\frac{m_s}{2} + (d-1)m_s} 
\\ 
& 
\le 
2^{d + \frac{m_s}{2} + (d-1)m_s }. 
\end{split} 
\end{equation} 
Since $V \in B_{2 m_s}$, 
the last sum in  \eqref{Eq:WSer-PS-4} 
can be estimated as follows: 
\begin{equation}
\label{Eq:WSer-PS-5} 
\begin{split}
\bigg| 
    \sum\limits_{ \mathbf{n} \in V } 
    \tau_{\mn} W_{\mn} ( \mg )  
\bigg| 
\le 
\# V  
\cdot 
\max\limits_{ \mathbf{n} \in B_{2 m_s} } 
| \tau_{\mn} | 
& 
\stackrel{\eqref{Eq:WSer-PS-5-1}, \, \eqref{Eq:Tau-Coeff-0}}{\le} 
2^{d + \frac{m_s}{2} + (d-1)m_s + 1} 
\cdot 
2^{ s - 1 - d m_s }  
\\ 
& 
= 
2^{d+s - \frac{m_s}{2} - 1 }. 
\end{split} 
\end{equation}  

Now fix an arbitrary $\boldsymbol{\sigma} 
\in \{ 0, 1 \}^d \setminus \{ \mathbf{0}, \mathbf{1} \}$ 
and estimate the sum $\sum_{ \mathbf{n} \in W_{\boldsymbol{\sigma}} }$ in \eqref{Eq:WSer-PS-4}. 
We denote $A_1 := \{ j \colon \sigma^j = 1 \}$ and $A_0 := \{ j \colon \sigma^j = 0 \}$. 
We divide the coordinates of each $d$-dimensional vector into two groups collecting in the natural order 

- in the first of them coordinates with indices from $A_1$,

- in the second --- from $A_0$. 

We mark vectors formed by first-group coordinates with a superscript $*$ and those formed by second-group coordinates with a subscript $*$.
In the new notation, the vector $\mathbf{n} \in W_{\boldsymbol{\sigma}}$ 
is $( \mathbf{n}^*, \mathbf{n}_* )$, where:  
\[ 
\mathbf{n}^* 
= 
2^{2m_s} \mathbf{1}^*
+ 
2^{m_s} \mathbf{M}^* 
+ 
2^{\frac{m_s}{2}} \mathbf{L}^* 
+ 
\mathbf{k}^*, 
\quad 
\mathbf{k}^* < \mathbf{K}^*;
\]
\[
\mathbf{n}_* 
= 
2^{2m_s} \mathbf{1}_*
+ 
2^{m_s} \mathbf{m}_* 
+ 
\mathbf{q}_*, 
\quad 
\mathbf{m}_* 
< 
\mathbf{M}_*, 
\quad 
\mathbf{q}_* 
< 
2^{m_s} \mathbf{1}.
\]
In the new notation, we write the sum 
$\sum\limits_{ \mathbf{n} \in W_{\boldsymbol{\sigma}} }
\tau_{\mn} W_{\mn} ( \mg )$, 
with coefficients taken from \eqref{Eq:Tau-Coeff-0}: 
\begin{equation} 
\label{Eq:BigSum}
\begin{split}
\sum\limits_{ \mathbf{n} \in W_{\boldsymbol{\sigma}} }
\tau_{\mn} W_{\mn} ( \mg ) 
= 
& 
2^{s - 1 - dm_s} R_{2m_s \mathbf{1}} (\mg)  
W_{ 2^{m_s} \mathbf{M}^* + 2^{\frac{m_s}{2}} \mathbf{L}^* } ( \mg^* ) 
\cdot I_1 \cdot I_2 \cdot I_3, 
\\ 
I_1
:= 
& 
\sum_{ \mathbf{k}^* < \mathbf{K}^* } 
W_{ \mathbf{q}^* \mathbf{M}^* }^{(m_s)} 
W_{ \mathbf{q}^* } ( \mg^* ), 
\\ 
I_2
:= 
&
\sum_{ \mathbf{m}_* < \mathbf{M}_* } 
W_{ \mathbf{m}_* } ( \mg_* ) 
\sum_{ \mathbf{q}_* < 2^{m_s} \mathbf{1}_* } 
W_{ \mq_* \mathbf{m}_* }^{(m_s)} 
W_{ \mq_* } ( \mg_* ),  
\\ 
I_3
:= 
& 
\mathrm{I} 
( \Delta^{(m_s)}_{ ( \mathbf{M}^*, \mathbf{m}_* ) } 
\subset \widetilde{F}_{s-1} ). 
\end{split}
\end{equation}
We estimate the values $| I_1 |$, $| I_2 |$ и $| I_3 |$ from above. 

Obviously, $| I_3 | \le 1$. 
Further, since $\mathbf{K}^* < 2^{\frac{m_s}{2} } \mathbf{1}^*$, 
the sum in $I_1$ contains at most $2^{\frac{m_s}{2} \# A_1}$ terms of $\pm 1$, 
whence $|I_1| \leq 2^{\frac{m_s}{2} \# A_1}$.

Let us take the vector  $\mathbf{r}_*$ such that  
$\Delta^{(m_s)}_{\mathbf{r}_*} \ni \mg_*$  
and then we get $W_{ \mq_* } ( \mg_* ) = W_{ \mq_* \mathbf{p}_* }^{(m_s)}$,  
\begin{equation} 
\label{Eq:I2}
\sum_{ \mathbf{q}_* < 2^{m_s} \mathbf{1}_* } 
W_{ \mq_* \mathbf{m}_* }^{(m_s)} 
W_{ \mq_* \mathbf{r}_* }^{(m_s)} 
\stackrel{\eqref{Eq:Walsh-Matrix-0}}{=} 
2^{m_s \# A_0} 
\delta^{ \mathbf{r}_* }_{ \mathbf{m}_* }, 
\qquad 
I_2 = 2^{m_s \# A_0} W_{ \mathbf{r}_* } ( \mg_* ), 
\end{equation} 
hence $| I_2 | = 2^{m_s \# A_0}$.  

Putting together \eqref{Eq:WSer-PS-4}, \eqref{Eq:WSer-PS-5}, and \eqref{Eq:BigSum},  
as well as the estimates for $| I_1 |$, $| I_2 |$ and $| I_3 |$, we obtain 
\[ 
\begin{split}
| E | 
& 
\le 
2^{d+s-1 - \frac{m_s}{2} } 
+ 
( 2^d-2 ) 
\cdot 
2^{s - 1 - dm_s}
\cdot 
2^{\frac{m_s}{2} \# A_1 } 
\cdot 
2^{m_s \# A_0}  
\\ 
& 
\le 
2^{d + s - \frac{m_s}{2} } \to 0 
\quad 
\text{as $s \to \infty$}   
\end{split} 
\]
(we used the fact that $\# A_1 + \# A_0 = d$ and $\# A_1 \ge 1$). 
Consequently, taking into account \eqref{Eq:E}, 
we get $S_{\mathbf{N}} ( \mathbf{g} ) \rightarrow 0$ 
as $\min N^i \to \infty$. 
Thus, the series $\sum\limits_{ \mn = \mathbf{0} }^\infty
\tau_{\mn} W_{\mn} ( \mg )$ converges to zero over rectangles for all $\mathbf{g} \in \mG^d \setminus F$. 
Consequently, for the same $\mathbf{g}$, it converges to zero over cubes. 
Since the indices of all nonzero coefficients of this series are located 
in the binary blocks $B_k$, 
at the points $\mathbf{g} \in \mG^d \setminus F$, the convergence to zero over rectangles, 
together with Proposition~\ref{Prop:WSer-Num}, 
yields iterated convergence to zero at these points for any order of iterated summation.
The proof is complete. 
\end{proof}

\begin{theorem} 
\label{T:Main-1-1}
Let the set $F^{\boldsymbol{\pi}}$ be defined by formulas 
\eqref{Eq:Fspi}--\eqref{Eq:Fpi}, 
where $\boldsymbol{\pi} \in S^{\prime}$ 
$($recall that $S^{\prime}$ consists of $\boldsymbol{\pi} = ( \boldsymbol{\pi}_s, \, s \in \mathbb{N} )$ satisfying condition \eqref{Eq:pis-1}, see \S~\ref{S:M-Set}$)$. 
Then the set $F^{\boldsymbol{\pi}}$ 
is an $M$-set for the $d$-dim Walsh system 
under convergence over rectangles, 
cubes, or under iterated convergence for any order of summation.  
Moreover the $d$-multiple Walsh series $\sum\limits_{ \mn = \mathbf{0} }^\infty 
\tau^{\boldsymbol{\pi}}_{\mn} W_{\mn} ( \mg )$, 
generating the quasi-measure $\tau_{F^{\boldsymbol{\pi}}}$,  
is a null-series that realizing the $M$-set $F^{\boldsymbol{\pi}}$. 
\end{theorem}

\begin{proof}
The proof essentially repeats the previous one. However, 
we will go through it step-by-step, highlighting the common features while explaining why, 
in contrast to Lemmas~\ref{Lem:2-1}--\ref{Lem:2-3} and Remarks~\ref{Rem:2-1}--\ref{Rem:2-3}, 
the theorem is formulated specifically for sequences $\boldsymbol{\pi}$ 
of permutations from $S^{\prime}$ rather than from $S$.

The proof of Theorem~\ref{T:Main-1} can be roughly divided into three main stages. 

At the first stage, we established formula~\eqref{Eq:WSer-PS-3}, which remains valid in our case since it relies on Corollary~\ref{Cor:Main-4}. This corollary holds for all quasi-measures $\tau_F$ constructed in Section~\ref{Subsubs:Tau-F} for non-empty closed sets $F$, including in particular the sets $F^{\boldsymbol{\pi}}$.

The second stage consisted of proving the estimate \eqref{Eq:WSer-PS-5}, 
which remains unchanged if we replace the coefficients $\tau_{\mn}$ with 
$\tau^{\boldsymbol{\pi}}_{\mn}$.  
Indeed, in the new setting, neither $\#V$ nor $\max\limits_{\mathbf{n} \in B_{2m_s}} |\tau_{\mathfrak{n}}|$, which equals $2^{s-1-dm_s}$, will change (compare \eqref{Eq:Tau-Coeff-0} and \eqref{Eq:Tau-Coeff-16}). 

The third stage involved establishing upper bounds for the quantities $|I_1|$, $|I_2|$, and $|I_3|$ appearing in \eqref{Eq:BigSum}. 
The estimates for $|I_1|$ and $|I_3|$ transfer immediately to the new setting. 
Since the estimate for $|I_2|$ employs formula \eqref{Eq:I2}, where the coordinates of all vectors are split into two groups, in the new setting formula \eqref{Eq:I2} remains valid only when the permutations $\boldsymbol{\pi}_s$ act ``coordinate-wise'', i.e., satisfy \eqref{Eq:pis-1}.
Thus, under the restriction $\boldsymbol{\pi} \in S^{\prime}$ to the appropriate index subset, the estimate for $|I_2|$ remains valid in the modified setting.
   
From the above arguments, the statement of the theorem follows.
\end{proof}

\begin{theorem} 
\label{T:Main-2-2-2}
Let $F$ be a $M$-set for the $d$-dim Walsh system while 
$\sum\limits_{\mathbf{n}=0}^{\infty} \tau_{\mathbf{n}} W_{\mathbf{n}}$ be 
the $d$-multiple Walsh null-series 
constructed in Theorem~\ref{T:Main-1}. 
If $\sum\limits_{\mathbf{n}=0}^{\infty} \psi_{\mathbf{n}} W_{\mathbf{n}}$ is 
another $d$-multiple Walsh null-series 
converging to zero by rectangles or cubes outside the set $F$,  
and $\psi_{\mathbf{n}} = o(\tau_{\mathbf{n}})$ as $\max n_j \to \infty$, 
then all $\psi_{\mathbf{n}} = 0$.
\end{theorem}
\begin{proof}
Let $\psi$ be the quasi-measure 
generated by the series $\sum\limits_{\mn=0}^{\infty} \psi_\mn W_\mn $, 
$\psi_\mn = \widehat{\psi}_\mn$. 
If not all $\widehat{\psi}_{\mathbf{n}}$ are zero, 
then the quasi-measure $\psi$ is non-trivial 
and there exists a dyadic cube $\Delta^{(m_{s_0})}$ (of rank $m_{s_0}$) such that
$\psi ( \Delta^{ ( m_{s_0} ) } ) = C \ne 0$. 
Here $(m_s)$ is the increasing sequence of natural numbers introduced in Section~\ref{S:M-Set} for constructing the set $F$.

We claim that in can be found a dyadic cube  
$\Delta^{ ( m_{s_0+1} ) } \su \Delta^{ ( m_{s_0} ) }$ 
(of rank $m_{s_0+1}$), 
for which 
\begin{equation} 
\label{Eq:1111}  
| \psi ( \Delta^{ ( m_{s_0+1} ) } ) | 
\ge \dfrac{ 2 | C | }{ 2^{ -d ( m_{s_0+1} - m_{s_0} ) } }. 
\end{equation} 
Indeed, there exist $2^{d(m_{s_0+1} - m_{s_0})}$ 
dyadic cubes $I \subset \Delta^{(m_{s_0})}$ of rank $m_{s_0+1}$ and precisely half of them are disjoint from $F$, by the construction of this set, 
and thus $\psi(I) = 0$ for all cubes $I$ from this half.
If for all cubes $I$ from the other half 
the inequality
\[
| \psi ( I )| 
< 
\frac{ 2 | C | }{ 2^{ d ( m_{s_0+1} - m_{s_0} ) } }  
\] 
holds, then by the additivity of the quasi-measure $\tau$ 
we get the relation  
\[ 
| \psi ( \Delta^{ ( m_{s_0} ) } ) | 
< 
\dfrac{ 2^{d ( m_{s_0+1} - m_{s_0} )} }{2} \, 
\frac{ 2 | C | }{ 2^{ d ( m_{s_0+1} - m_{s_0} ) } } 
= 
C,   
\]
which leads to a contradiction with $\psi ( \Delta^{ ( m_{s_0} ) } ) = C$. 
Thus, the desired cube $\Delta^{(m_{s_0+1})}$ exists. 

Iterating this process, 
we inductively  construct a nested sequence of dyadic cubes
$\Delta^{ ( m_{s_0} ) } \supset \Delta^{ ( m_{s_0+1} ) } \supset \ldots \Delta^{ ( m_{s_0+1} ) } \supset \ldots$ 
such that the  rank $\Delta^{ ( m_{s_0+1} ) }$ equals $m_{s_0+k}$ and  
\begin{equation} 
\label{Eq:11111}
| \psi ( \Delta^{ ( m_{s_0+k} ) } ) | 
\ge 
\frac{ 2^k | C | }{2^{ d ( m_{s_0 + k} - m_{s_0} ) }}. 
\end{equation}

Now fix a natural number $k$, 
set $s := s_0 + k$, and estimate 
$\left| \psi \left( \Delta^{(m_s)} \right) \right|$ 
from above.
Observe that $\psi ( \Delta ) = 0$ holds for all dyadic cubes $\Delta \subset \mG^d \setminus F$, since the series $\sum\limits_{\mn=0}^{\infty} \psi_\mn W_\mn$ 
converges to zero over cubes outside $F$ (see item~\ref{Subsub:F-Tau-05}).
In view of this fact, we derive
\begin{equation} 
\label{Eq:12} 
\begin{split}
\psi ( \Delta^{(m_s)} )  
& 
= 
\sum\limits_{\Delta^{(2m_s+1)} \subset \widetilde{F}_s \cap \Delta^{(m_s)}} 
\psi( \Delta )   
\\ 
& 
\stackrel{\eqref{Eq:Canon-Iso}}{=} 
\sum\limits_{\Delta^{(2m_s+1)} \subset \widetilde{F}_s \cap \Delta^{(m_s)}}
\sum_{ \mn < 2^{2m_s+1} \mathbf{1} } 
\int\limits_{\Delta} \psi_{\mn} W_{\mn} \, d \mu, 
\\ 
& 
= 
\sum_{ \mn < 2^{2m_s+1} \mathbf{1} } 
\psi_{\mn} 
\int\limits_{ F_s \cap \Delta^{(m_s)} } W_{\mn} \, d \mu.  
\end{split}
\end{equation} 

The integral on the right-hand side of \eqref{Eq:12} 
is $2^{-s} \widehat{\tau}_{\mn} ( \Delta^{(m_s)} )$ 
and vanishes 
if the vector $\mathbf{n}$ satisfies condition of Lemma~\ref{Lem:2-2}
(we use Lemmas~\ref{Lem:2-2} и \ref{Lem:2-3}). 
Therefore, the sum on the right-hand side of \eqref{Eq:12} 
can be written as
$= \sum_{ \mn < 2^{m_s} \mathbf{1} } 
+ \sum_{ 2^{m_s} \mathbf{1} \le \mn < 2^{2m_s+1} \mathbf{1} }$. 
Next, 
if $\mn < 2^{m_s} \mathbf{1}$, 
then the Walsh function  $W_{\mn}$ is constant on $\Delta^{(m_s)}$  
and $\mu ( F_s \cap \Delta^{(m_s)} ) = \mu ( \Delta^{(m_s)} )/ 2$, 
which yields
\[ 
\int\limits_{ F_s \cap \Delta^{(m_s)} } W_{\mn} \, d \mu
= 
\dfrac{1}{2}
\int\limits_{ \Delta^{(m_s)} } W_{\mn} \, d \mu. 
\]

Taking into account the preceding arguments, we obtain  
\[ 
\begin{split}
\left| \psi ( \Delta^{(m_s)} ) \right| 
& 
\le 
\bigg| 
    \sum_{ \mn < 2^{m_s} \mathbf{1} } 
    \frac{1}{2} 
    \int\limits_{\Delta^{(m_s)}} 
    \psi_{\mn} W_{\mn} (t) \, d \mu
\bigg| 
+ 
\bigg| 
    2^{-s} 
    \sum_{ 2^{2m_s} \mathbf{1} \le \mn < 2^{2m_s+1} \mathbf{1} } 
    \psi_{\mn} 
    \widehat{\tau}_{\mn} ( \Delta^{(m_s)} ) 
\bigg| 
\\
& 
\le 
\frac{1}{2} \left| \psi ( \Delta^{(m_s)} ) \right| 
+ 
2^{-s}
\max_{ \mn \in B_{2m_s} } 
| \psi_\mn | 
\sum_{ 2^{2m_s} \mathbf{1} \le \mn < 2^{2m_s+1} \mathbf{1} } 
| \widehat{\tau}_{\mn} ( \Delta^{(m_s)} ) |
\\ 
& 
= 
\frac{1}{2} \left| \psi ( \Delta^{(m_s)} ) \right| 
+ 
2^{-s}
\max_{ \mn \in B_{2m_s} } 
| \psi_\mn | 
\,
2^{d m_s}
2^{s - 1 - d m_s}.  
\end{split}
\]  
In the last step, we used the fact that 
if $2^{2m_s} \mathbf{1} \le \mn < 2^{2m_s+1} \mathbf{1}$, 
then $\mn$  has the form  \eqref{Eq:2} 
and, according to \eqref{Eq:Tau-Coeff-000}, 
the quantity
$| \widehat{\tau}_{\mn} ( \Delta^{(m_s)} ) |$ 
equals $2^{s - 1 - d m_s}$ 
for exactly one $\mmp$, 
i.e., for $2^{d m_s}$ vectors  $\mn$,  
while vanishing for all others. 
From the last chain of inequalities (recalling that $s := s_0 + k$), we get  
\[ 
\begin{split}
\max_{\mn \in B_{ 2m_s } } | \psi_\mn | 
\ge  
\big| \psi ( \Delta^{(m_s)} ) \big|  
& 
\stackrel{\eqref{Eq:11111}}{\ge}  
2^{s-s_0-d m_s +d m_{s_0}} | C | 
\\
&
= 
2^{s - 1 - d m_s} 
2^{1-s_0+ d m_{s_0}} | C | 
= 
2^{1-s_0+ d m_{s_0}} | C | 
\max_{\mn \in B_{ 2m_s } } | \tau_\mn |.  
\end{split}
\]
It follows that $\psi_{\mn}$ is not $o(\tau_{\mn})$ 
as $\max n^j \to \infty$. This contradicts the theorem's assumption 
and completes the proof.
\end{proof}


The following theorem~\ref{T:U-1} is an analogue of Theorem~2 from~\cite{plotnikov-2010}.


\begin{theorem} 
\label{T:U-1}
Let a sequence $( \mathbf{n}_i \in \mbN^d, \, i \in \mbN )$,  
a $d$-dim Walsh series $(S)$, and $\tau$ be the quasi-measure generated by the series $(S)$ be given.  
Suppose that each $\mathbf{n}_i$ lies in some binary block $B_{k_i}$,  
with $\lim\limits_{i \to \infty} k(i) = \infty$,  
and at least one of the following conditions is satisfied. 

$(\mathrm{a})$ 
The series $(S)$ $2$-converges at least at one point of the group $\mathbb{G}^d$.

$(\mathrm{b})$ 
$\mathbf{n}_i = n_i \mathbf{1}$ for all $i$ and the series $(S)$ is $\lambda$-convergent 
for at least one $\lambda > 1$ at at least one point of the group $\mathbb{G}^d$.
 
Then 
\begin{equation} 
\label{Eq:U-00}
\lim\limits_{i \to \infty} 
\int\limits_{\Delta} W_{\mathbf{n}_i} d \tau = 0 
\end{equation}
for all dyadic cube $\Delta$. 
\end{theorem}

\begin{proof} 
Let us take and fix an arbitrary dyadic cube $\Delta$, and let $q$ be its rank.
$\Delta$ is a certain coset $\mathbf{g}_0 \oplus \Delta^{(q)}_{\mathbf{0}}$, 
$\mathbf{g}_0 \in \mathbb{G}^d$.  
If $\mathbf{n}i \in B{k_i}$ and $k_i \ge q$, then the integral from \eqref{Eq:U-00} can be expressed in terms of the coefficients $a_{\mathbf{n}}$ of the series $(S)$ using the formula
\begin{equation} 
\label{Eq:U-0001}
\int\limits_{\Delta} W_{\mathbf{n}_i} d \tau 
= 
2^{-qd} 
\sum\limits_{\mathbf{n} < 2^q \mathbf{1}} 
a_{ \mathbf{n}_i \oplus \mathbf{n} } 
W_{ \mathbf{n} } ( \mathbf{g}_0 ).  
\end{equation} 
Formula \eqref{Eq:U-0001} is essentially derived in the proof of Theorem 5.1 from paper \cite{MP-EMJ-2019}, 
which studied systems of characters of a zero-dimensional compact Abelian group, a special case of which is the Walsh system.
Here, $\oplus$ denotes dyadic addition of vectors from $( \mathbb{N}_0 )^d$, which is defined as
\[ 
\mathbf{n} \oplus \mathbf{m} 
= 
( n^1 \oplus m^1, \ldots, n^d \oplus m^d ), 
\qquad 
\mathbf{n}, \, \mathbf{m} 
\in ( \mathbb{N}_0 )^d, 
\] 
and for $d=1$, for nonnegative integers $n$ and $m$ with binary coefficients $n_k$ and $m_k$, 
the operation $\oplus$ is bitwise addition modulo 2:   
$n \oplus m := \sum_{k=0}^\infty | n_k - m_k | 2^k$.

Assume that condition (a) of the theorem holds. 
Then the coefficients of the series $(S)$ with indexes from the binary blocks tend to zero in the following sense:
\begin{equation} 
\label{Eq:U-0002}
\lim\limits_{k \to \infty} 
\sup\limits_{\mathbf{m} \in B_k} 
| a_{ \mathbf{m} } | 
= 0  
\end{equation} 
(see, e.g., \cite{plotnikov-2017, MP-EMJ-2019}).  
If $\mathbf{n}_i \in B_{k_i}$ where $k_i \ge q$, and $\mathbf{n} < 2^q \mathbf{1}$, 
then $\mathbf{n}_i \oplus \mathbf{n} \in B_{k_i}$.   
Therefore, it follows from formula \eqref{Eq:U-0002} 
that the right-hand side of \eqref{Eq:U-0001} tends to zero as $i \to \infty$.  
We prove equality \eqref{Eq:U-00} in this case.

Now suppose condition (b) of the theorem holds. 
From the $\lambda$-convergence of the series $(S)$ at some point $\mathbf{g} \in \mathbb{G}^d$ for $\lambda > 1$ 
and the obvious formula
\[ 
a_{\mathbf{m}} W_{\mathbf{m}} (\mathbf{g})
= 
\sum\limits_{ \boldsymbol{\sigma} \in \{ 0,1\}^d } 
(-1)^{| \sigma^1 + \ldots + \sigma^d | }
S_{\mathbf{m} + \mathbf{1} - \boldsymbol{\sigma} } (\mathbf{g}),    
\]
for any $\rho \in (1, \lambda )$, the relation  
\begin{equation} 
\label{Eq:U-0003}
\lim a_{\mathbf{m}} = 0 
\quad 
\text{as $\max\limits_j \{ m^j \} \to \infty$ 
and $\mathbf{m} \in M_{\rho} := \big\{ \mathbf{m} \colon \max_{j,k} \{ m^j / m^k \} \le \rho \big\}$}  
\end{equation} 
holds.
Since $\mathbf{n}_i = n_i \mathbf{1}$, for large $i$ all indices $\mathbf{n}_i \oplus \mathbf{n}$ with $\mathbf{n} < 2^q \mathbf{1}$ lie in $M_{\rho}$. Therefore, it follows 
from \eqref{Eq:U-0003} that the right-hand side of \eqref{Eq:U-0001} tends to zero as $i \to \infty$, 
which proves \eqref{Eq:U-00} in this case as well.
\end{proof}


\begin{theorem}
\label{T:U-2}
Assume the sequence $( \mathbf{n}_i \in \mbN^d, \, i \in \mbN )$ satisfies  
the conditions of Theorem ~\ref{T:U-1}. 
Then the set
\begin{equation} 
\label{Eq:U-05} 
F := \{ \mathbf{g} \in \mathbb{G}^d \colon W_{ \mathbf{n}_i} ( \mathbf{g} ) = 1 
\;\; \text{for all $i$} \} 
\end{equation}
is $U$-set for the $d$-dim Walsh system under $2$-convergence 
and if, additionally, $\mathbf{n}_i = n_i \mathbf{1}$, then also under $\lambda$-convergence with any $\lambda > 1$.
\end{theorem}

\begin{proof} 
Suppose there exists a nontrivial series \eqref{Eq:WSer} that is $2$-convergent to zero outside $F$. 
Then, by Theorem~\ref{T:U-1}, the quasi-measure $\tau$ generated by this series satisfies condition \eqref{Eq:U-00}.

On the other hand, the quasi-measure $\tau$ is not identically zero (just like the original series), 
hence there exists a dyadic cube $\Delta$ such that  $\tau(\Delta) \ne 0$. 
Let $q$ be its rank and $k_i > q$.
We obtain
\begin{equation} 
\label{Eq:U-06} 
\begin{split} 
\int\limits_{\Delta} 
W_{ \mathbf{n}_i } d \tau 
& 
= 
\sum_{ \Delta^{(k_j)} \subset \Delta } 
\int\limits_{ \Delta^{(k_j)} } 
W_{ \mathbf{n}_i } d \tau 
= 
\sum_{ ( \Delta^{(k_j)} \subset \Delta ) \; \wedge \; ( \Delta^{(k_j)} \cap F \neq \emptyset )} 
\int\limits_{ \Delta^{(k_j)} } 
W_{ \mathbf{n}_i } d \tau  
\\ 
& 
\stackrel{\eqref{Eq:U-05}}{=} 
\sum_{ ( \Delta^{(k_j)} \subset \Delta ) \; \wedge \; ( \Delta^{(k_j)} \cap F \neq \emptyset ) } 
\tau ( \Delta^{(k_j)} ) 
\\ 
& 
= 
\tau(\Delta) 
- 
\sum_{ (\Delta^{(k_j)} \subset \Delta ) \; \wedge \; ( \Delta^{(k_j)} \cap F = \emptyset ) } 
\tau ( \Delta^{(k_j)} ) 
\stackrel{ \text{п.~\ref{Subsub:F-Tau-05}} }{=} 
\tau(\Delta) \neq 0. 
\end{split}
\end{equation}
\eqref{Eq:U-06} yields a contradiction with \eqref{Eq:U-00}, which proves the main part of the theorem's statement.

If, additionally, $\mathbf{n}_i = n_i \mathbf{1}$, then, repeating the reasoning above with any $\lambda > 1$, we again arrive at a contradiction with formula \eqref{Eq:U-00}, which also holds in this case according to Theorem~\ref{T:U-1}. 
The theorem is proved.
\end{proof} 

The sets \eqref{Eq:U-05} are the $d$-dim version of the Dirichlet sets for the Walsh system. 
In the univariate case, such sets were studied in uniqueness theory by K.~Yoneda~\cite{yoneda-1982, yoneda-1984}, 
who established in particular~\cite{yoneda-1982} that all one-dim Dirichlet sets are $U$-sets for the Walsh system.


\begin{cor} 
\label{Cor:U-2}
Let $(m_s)$ be the sequence from Section \S~\ref{S:M-Set},  
and let $F_s$ denote the union of ``graphs'' of the same $d$-dim Walsh function  
with index $\mathbf{n}$ satisfying  
$2^{m_s} \mathbf{1} \leq \mathbf{n} < 2^{m_s+1} \mathbf{1}$,  
compressed to the dyadic cubes of rank $m_s$.
Then the set $F=\cap_{s=1}^{\infty}F_s$ is a $U$-set
for the $d$-dim Walsh system for $2$-convergence and if $\mathbf{n} = n \mathbf{1}$, 
where $2^{m_s} \le n \le 2^{m_{s}+1} - 1$, then also for $\lambda$-convergence with any $\lambda > 1$.
\end{cor} 

\begin{proof}
The set $F$ can be expressed as  
$F = \bigcap\limits_{s=1}^\infty F_s$,  
where each $F_s$ comprises all $\mathbf{g} \in \mathbb{G}^d$  
satisfying $W_{\mathbf{n}_s}(\mathbf{g}) = 1$,  
with indices $\mathbf{n}_s = 2^{2m_s} \mathbf{1} + 2^{m_s} \mathbf{n}$.  
It is straightforward to verify that  
the sequence $(\mathbf{n}_s)$ satisfies  
the conditions of Theorem~\ref{T:U-1}.
It remains to apply Theorem~\ref{T:U-2}.  
\end{proof}


\begin{theorem} 
\label{T:Portion-2} 
Every non-empty set of the form $G \cap F^{\boldsymbol{\pi}} \neq \emptyset$, 
where $\boldsymbol{\pi} \in S^{\prime}$, and $G \subset \mathbb{G}^d$ is open, 
is an $M$-set for the $d$-dimensional Walsh system under rectangular convergence, 
cubic convergence, or iterated one with any order of summation. 
Moreover, there exists a dyadic cube $\Delta^{(m_{s_0})} \subset G$ 
such that the $d$-dimensional Walsh series \eqref{Eq:WSer}, generating the quasi-measure $\tau_{F^{\boldsymbol{\pi}}}|_{\Delta^{(m_{s_0})}}$, realizes the $M$-set 
$G \cap F^{\boldsymbol{\pi}}$.
\end{theorem}

\begin{proof}  
Since any open set $G \subset \mathbb{G}^d$ is at most a countable union of dyadic cubes, there exists a dyadic cube $\Delta^{(m_{s_0})} \subset G$ such that $\Delta^{(m_{s_0})} \cap F^{\boldsymbol{\pi}} \neq \emptyset$.

The rest of the proof follows the same lines as the proofs of Theorems ~\ref{T:Main-1} and \ref{T:Main-1-1}.
Thus, for points $\mathbf{g} \notin \Delta^{(m_{s_0})} \cap F^{\boldsymbol{\pi}}$, 
formula~\eqref{Eq:WSer-PS-3} remains valid, since even with the new settings 
we are operating under the conditions of Corollary~\ref{Cor:Main-4}. 
Furthermore, estimate \eqref{Eq:WSer-PS-5} remains valid because the Fourier coefficients of the restricted quasi-measure $\tau_{F^{\boldsymbol{\pi}}}|_{\Delta^{(m_{s_0})}}$ are majorized by (Remark~\ref{Rem:2-3}) the Fourier coefficients of the quasi-measure $\tau_{F^{\boldsymbol{\pi}}}$.
Finally, the estimates for the analogues of the quantities $|I_1|$, $|I_2|$, and $|I_3|$ from \eqref{Eq:BigSum} clearly carry over to the new setting (see the reasoning in the proof of Theorem~\ref{T:Main-1-1}). 
The theorem's statement follows from the preceding arguments.
\end{proof}



\begin{thebibliography}{99} 
\bibitem{men'shov}
D.~E.~Menshov, \emph{Selected Works. Mathematics}, Factorial, M., 1997 (Russian).

\bibitem{zygmund-2002}
A.~Zygmund, \emph{Trigonometric Series. Vol. I}, Cambridge Univ. Press, Cambridge, 2002.

\bibitem{bari-1961}
N.~K.~Bari, \emph{A treatise on trigonometric series}, Pergamon Press, New York, 1964.

\bibitem{kechris-louveau-1987}
A.~Kechris, A.~Louveau, \emph{Descriptive Set Theory and the Structure of Sets of Uniqueness}, Cambridge Univ. Press, Cambridge, 1987.

\bibitem{kholshchevnikova-2019}
N.~Kholshchevnikova, ``The union problem and the category problem of sets of uniqueness in the theory of orthogonal series,'' \emph{Real Anal. Exchange}, vol. 44, no. 1, pp. 65--76, 2019.

\bibitem{kozma-olevskii-2020}
G.~Kozma, A.~M.~Olevskii, ``Cantor uniqueness and multiplicity along subsequences,'' \emph{St. Petersburg Math. J.}, vol. 32, no. 2, pp. 261--277, 2021.

\bibitem{gevorkyan-2020}
G.~G.~Gevorkyan, ``Uniqueness theorems for one-dimensional and double Franklin series,'' \emph{Izv. Math.}, vol. 84, no. 5, pp. 829--844, 2020.

\bibitem{plotnikov-2020}
M.~Plotnikov, ``On the Vilenkin--Chrestenson systems and their rearrangements,'' \emph{J. Math. Anal. Appl.}, vol. 492, no. 1, 2020.

\bibitem{skvortsov-2021}
V.~A.~Skvortsov, ``Reconstruction of a Generalized Fourier Series from Its Sum on a Compact Zero-Dimensional Group in the Non-Abelian Case,'' \emph{Math. Notes}, vol. 109, no. 4, pp. 630--637, 2021.

\bibitem{wronicz-2021}
Z.~Wronicz, ``Uniqueness of series in the Franklin system and the Gevorkyan problems,'' \emph{Opusc. Math.}, vol. 41, no. 2, pp. 269--276, 2021.

\bibitem{lukomskii-2021}
S.~F.~Lukomskii, ``On the Uniqueness Sets of Multiple Walsh Series for Convergence in Cubes,'' \emph{Math. Notes}, vol. 109, no. 3, pp. 427--434, 2021.

\bibitem{plotnikov-2022}
M.~G.~Plotnikov, ``Uniqueness sets of positive measure for the trigonometric system,'' \emph{Izv. Math.}, vol. 86, no. 6, pp. 1179--1203, 2022.

\bibitem{gevorkyan-2023}
G.~G.~Gevorkyan, ``On uniqueness for Franklin series with a convergent subsequence of partial sums,'' \emph{Sb. Math.}, vol. 214, no. 2, pp. 197--209, 2023.

\bibitem{gevorkyan-2024}
G.~G.~Gevorkyan, ``On uniqueness for series in the general Franklin system,'' \emph{Sb. Math.}, vol. 215, no. 3, pp. 308--322, 2024.

\bibitem{keryan-khachatryan-2024}
K.~A.~Keryan, A.~L.~Khachatryan, ``A uniqueness theorem for orthonormal spline series,'' \emph{Acta Math. Hungarica}, vol. 174, pp. 20--48, 2024.

\bibitem{kazakova-plotnikov-2025-1}
A.~D.~Kazakova, M.~G.~Plotnikov, ``Sets of uniqueness for subsystems of the trigonometric system,'' \emph{Math. Notes}, vol. 117, no. 1, pp. 75--84, 2025.

\bibitem{kazakova-plotnikov-2025-2}
A.~D.~Kazakova, M.~G.~Plotnikov, ``On Lacunarity and Uniqueness for p-adic Analogs of Rademacher Chaos,'' \emph{Siberian Math. J.}, vol. 66, no. 5, pp. 1184--1194, 2025.

\bibitem{shneider-1949}
A.~A.~Shneider, ``On the uniqueness of expansions in Walsh functions,'' \emph{Sb. Math.}, vol. 24, no. 2, pp. 279--300, 1949.

\bibitem{schipp-1969}
F.~Schipp, ``Uber Walsh-Foureierreihen mit nichtnegativen Partialsummen,'' \emph{Ann. Univ. Sci. Budapest, Eotvos Sect. Math.}, vol. 12, pp. 43--48, 1969.

\bibitem{coury-1970}
J.~E.~Coury, ``A class of Walsh $M$-sets of mesure zero,'' \emph{J. Math. Anal. Appl.}, vol. 31, no. 2, pp. 318--320, 1970.

\bibitem{skvortsov-1976}
V.~A.~Skvortsov, ``An example of a zero series expansion in the Walsh system,'' \emph{Math. Notes}, vol. 19, no. 2, pp. 108--112, 1976.

\bibitem{skvortsov-1977-1}
V.~A.~Skvortsov, ``On the rate of approach to zero of the coefficients of null series in the Haar and Walsh systems,'' \emph{Math. USSR-Izv.}, vol. 11, no. 3, pp. 665--676, 1977.

\bibitem{gevorkian-1988}
G.~G.~Gevorkian, ``On coefficients of null-series and on sets of uniqueness of trigonometric and Walsh systems,'' \emph{Analysis Mathematica}, vol. 14, pp. 219--251, 1988.

\bibitem{skvortsov-1977-2}
V.~A.~Skvortsov, ``The $h$-measure of $M$-sets for a Walsh system,'' \emph{Math. Notes}, vol. 21, no. 3, pp. 186--189, 1977.

\bibitem{skvortsov-1979}
V.~A.~Skvortsov, ``On null series with respect to a multiplicative system,'' \emph{Vestnik MGU. Ser. mat., mekh.}, vol. 34, no. 6, pp. 61--65, 1979.

\bibitem{tuzikova-1985}
I.~I.~Tuzikova, ``An example of a null-series with respect to an orthogonal multiplicative system of functions,'' \emph{Soviet Math. (Iz. VUZ)}, vol. 29, no. 5, pp. 61--69, 1985.

\bibitem{bokaev-nurhanov-1993}
N.~A.~Bokaev, M.~A.~Nurkhanov, ``Example of a null-series with respect to periodic multiplicative systems,'' \emph{Math. Notes}, vol. 54, no. 6, pp. 1187--1191, 1993.

\bibitem{kholshchevnikova-2013}
N.~N.~Kholshchevnikova, ``Countably multiple null series,'' \emph{Proc. Steklov Inst. Math.}, vol. 280, pp. 280--291, 2013.

\bibitem{lukomskii-89}
S.~F.~Lukomskii, ``On certain classes of sets of uniqueness of multiple Walsh series,'' \emph{Math. USSR-Sb.}, vol. 67, no. 2, pp. 393--401, 1990.

\bibitem{gogoladze-2008}
L.~D.~Gogoladze, ``On the problem of reconstructing the coefficients of convergent multiple function series,'' \emph{Izv. Math.}, vol. 72, no. 2, pp. 283--290, 2008.

\bibitem{zherebeva-2009}
T.~A.~Zherebyova, ``A class of sets of uniqueness for multiple Walsh series,'' \emph{Moscow Univ. Math. Bull.}, vol. 64, no. 2, pp. 55--61, 2009.

\bibitem{skvortsov-tulone-2015}
V.~A.~Skvortsov, F.~Tulone, ``Multidimensional dyadic Kurzweil--Henstock- and Perron-type integrals in the theory of Haar and Walsh series,'' \emph{J. Math. Anal. Appl.}, vol. 421, no. 2, pp. 1502--1518, 2015.

\bibitem{Lukomskii-1992}
S.~F.~Lukomskii, ``On a U-set for multiple Walsh series,'' \emph{Analysis Mathematica}, vol. 18, no. 2, pp. 127--138, 1992.

\bibitem{plotnikov-07-1}
M.~G.~Plotnikov, ``On Uniqueness Sets for Multiple Walsh Series,'' \emph{Math. Notes}, vol. 81, no. 2, pp. 234--246, 2007.

\bibitem{plotnikov-07-2}
M.~G.~Plotnikov, ``On multiple Walsh series convergent over cubes,'' \emph{Izv. Math.}, vol. 71, no. 1, pp. 57--73, 2007.

\bibitem{plotnikov-2010}
M.~G.~Plotnikov, ``Quasi-measures on the group $G^m$, Dirichlet sets, and uniqueness problems for multiple Walsh series,'' \emph{Sb. Math.}, vol. 201, no. 12, pp. 1837--1862, 2010.

\bibitem{plotnikov-2017}
M.~G.~Plotnikov, ``$\lambda$-Convergence of Multiple Walsh--Paley Series and Sets of Uniqueness,'' \emph{Math. Notes}, vol. 102, no. 2, pp. 268--276, 2017.

\bibitem{ash-freiling-rinne-1993}
J.~M. Ash, C.~Freiling, D.~Rinne, ``Uniqueness of rectangularly convergent trigonometric series,'' \emph{Ann. Math.}, vol. 137, no. 1, pp. 145--166, 1993.

\bibitem{kholshchevnikova-2002}
N.~N.~Kholshchevnikova, ``Union of sets of uniqueness for multiple Walsh and multiple trigonometric series,'' \emph{Sb. Math.}, vol. 193, no. 4, pp. 609--633, 2002.

\bibitem{MP-EMJ-2019}
M.~Plotnikov, ``$V$-sets in the products of zero-dimensional compact Abelian groups,'' \emph{Eur. Math. J.}, vol. 5, pp. 223--240, 2019.

\bibitem{yoneda-1982}
K.~Yoneda, ``Perfect sets of uniqueness on the group $2^{\omega}$,'' \emph{Canad. J. Math.}, vol. 34, no. 3, pp. 759--764, 1982.

\bibitem{Schipp-and-Co}
F.~Schipp, W.~R.~Wade, P.~Simon, \emph{Walsh Series. An Introduction to Dyadic Harmonic Analysis}, Academiai Kiado, Budapest, 1990.

\bibitem{golubov-efimov-skvortsov}
B.~I.~Golubov, A.~V.~Efimov and V.~A.~Skvortsov, \emph{Walsh series and transforms. Theory and applications}, Math. Appl. (Soviet Ser.), vol. 64, Kluwer Acad. Publ., Dordrecht, 1991.

\bibitem{VA-Hung-2004}
V.~A.~Skvortsov, ``Henstock--Kurzweil type integrals in $P$-adic harmonic analysis,'' \emph{Acta Math. Acad. Paedagog. Nyh\'azi (N.S.)}, vol. 20, pp. 207--224, 2004.

\bibitem{plotnikov-2004}
M.~G.~Plotnikov, ``Uniqueness Questions for Some Classes of Haar Series,'' \emph{Math. Notes}, vol. 75, no. 3, pp. 360--371, 2004.

\bibitem{yoneda-1984}
K.~Yoneda, ``Dirichlet sets and some uniqueness theorems for Walsh series,'' \emph{Tohoku Math. J.}, vol. 38, no. 1, pp. 1--14, 1986.

\end{thebibliography}
\end{document}